\documentclass[12pt]{article}
\usepackage{graphicx} 
\usepackage[utf8]{inputenc}
\usepackage[T1]{fontenc}
\usepackage[letterpaper,margin=1in]{geometry}
\usepackage{enumerate}
\usepackage{amsmath,amssymb,amsthm}
\usepackage{xcolor}
\usepackage[colorlinks=true]{hyperref}
\usepackage{booktabs}
\usepackage[table]{xcolor}
\usepackage{algorithm, algorithmic,setspace}
\usepackage{subfigure}
\usepackage{multirow}
\newcommand{\ds}{\displaystyle}

\newcommand{\norm}[1]{\left\Vert#1\right\Vert}
\newcommand{\inner}[2]{\left\langle #1, #2 \right\rangle}

\newcommand{\paren}[1]{\left(#1\right)}
\newcommand{\parenc}[1]{\left[#1\right]}
\newcommand{\lag}{\mathcal{L}}
\newcommand{\R}{\mathbb{R}}
\newcommand{\N}{\mathbb{N}}
\newcommand{\Rcupinf}{\R\cup\{+\infty\}}
\newcommand{\qbox}[1]{\quad\hbox{#1}\quad}
\newcommand{\qqbox}[1]{\qquad\hbox{#1}\qquad}
\newcommand{\cali}{\mathcal}

\DeclareMathOperator{\argmin}{argmin}
\DeclareMathOperator{\prox}{prox}

\newtheorem{theorem}{Theorem}
\newtheorem{lemma}[theorem]{Lemma}
\newtheorem{proposition}[theorem]{Proposition}

\newtheorem{assumption}[theorem]{Assumption}

\newtheorem{remark}[theorem]{Remark}
\newtheorem{example}[theorem]{Example}

\usepackage{todonotes}

\allowdisplaybreaks

\title{Preconditioned primal-dual algorithms for saddle point problems: non-ergodic convergence rates}
\author{Huiyuan Guo$^{1}$ \and	Juan Jos\'e Maul\'en$^{2,3}$ \and
	Juan Peypouquet$^{1}$
}

\begin{document}
	
	\maketitle
	
	\begingroup

	\footnotetext[1]{\small Bernoulli Institute for Mathematics, Computer Science and Artificial Intelligence, University of Groningen, The Netherlands.  \texttt{hazel.guo@rug.nl}, \texttt{j.g.peypouquet@rug.nl}.}
	\footnotetext[2]{\small Instituto de Ciencias de la Ingeniería, Universidad de O’Higgins, Rancagua, Chile. \texttt{juan.maulen@postdoc.uoh.cl}.}
	\footnotetext[3]{\small Centro de Modelamiento Matem\'atico (CNRS IRL2807), Universidad de Chile, Santiago, Chile. }
	
	\endgroup

	\begin{abstract}
		We study a family of preconditioned primal dual algorithms for convex-concave saddle point problems by the dynamics introduced in \cite{apidopoulos2026preconditioned}.
		The proposed framework exploits the possible smooth + nonsmooth structure of the saddle point formulation. It includes, but is not limited to, linearly constrained convex optimization problems. The proposed antisymmetric preconditioners allow us to establish non ergodic convergence rates, accounting for possible computational errors in the implementation of the method. Finally, we present numerical experiments to indicate our well performed preconditioned primal dual algorithms.\\
		
		\textbf{Keywords:} Convex optimization $\cdot$ Saddle point problem $\cdot$ Primal dual algorithm
	\end{abstract}

	\section*{Introduction}
	
	Let $\mathcal{X}$ and $\mathcal{Y}$ be real Hilbert spaces, let $A:\mathcal{X}\to \mathcal{Y}$ be a bounded linear operator, and let $f:\mathcal X\to\Rcupinf$ and $g:\mathcal Y\to\Rcupinf$ be closed and convex. Consider the saddle point problem
	\begin{equation}\label{PD}
		\min_{x\in\mathcal X}\max_{y\in\mathcal Y}\lag(x,y),    
	\end{equation}
	where the {\it Lagrangian} $\lag$ is given by 
	$$\lag(x,y)=f(x)+\inner{Ax}{y}-g(y).$$
	This structure arises from various applications in machine learning \cite{bach2011convex}, PDEs \cite{attouch2014variational}, image processing \cite{chambolle2016introduction}, and optimal control \cite{o2013splitting}. Any convex optimization problem of the form
	\begin{equation} \label{P}
		\min_{x\in\mathcal X} f(x)+h(Ax),
	\end{equation}
	where $h$ is closed and convex, can be expressed as \eqref{PD} by setting $g=h^*$, the Fenchel conjugate of $h$. In particular, if $b\in\mathcal Y$ and $h$ is the indicator function of the set $\{b\}$, then \eqref{P} is the linearly constrained convex optimization problem
	\begin{equation} \label{COP}
		\min_{x\in\mathcal X}\big\{f(x):Ax=b\big\},
	\end{equation}
	with the corresponding Lagrangian $\lag (x,y)=f(x)+\inner{Ax-b}{y}$. The idea of introducing the constraints into the objective function by means of a coupling term dates back to Lagrange, but the actual implementation of the saddle point formulation \eqref{PD} in order to numerically approximate the solutions to \eqref{COP} was first proposed by Arrow and Hurwicz in \cite{hurwicz1951gradient}, and improved by Uzawa (see \cite{arrow1957gradient,arrow1958studies}). 
	
	{\bf The Primal-Dual Hybrid Gradient Algorithm.} The strategy outlined in \cite{hurwicz1951gradient} consists in approximately following the descent-ascent flow
	\begin{equation} \label{E:Arrow_Hurwicz_continuous}
		\left\{
		\begin{array}{rcl}
			-\dot x(t) & \in & \partial_x\lag\big(x(t),y(t)\big) \\
			\dot y(t) & \in & \partial_y\lag\big(x(t),y(t)\big)
		\end{array}
		\right.   \qquad\Longleftrightarrow\qquad 
		\left\{
		\begin{array}{rcl}
			\dot x(t) + \partial f\big(x(t)\big)+A^*y(t) & \ni & 0\\
			\dot y(t) + \partial g\big(y(t)\big)-Ax(t) & \ni & 0.
		\end{array}
		\right.   
	\end{equation}
	
	To this end, one can perform a semi-implicit discretization with step sizes $\sigma,\tau>0$ to obtain the {\it Primal Dual Hybrid Gradient} (PDHG) algorithm
	\[
	\left\{
	\begin{aligned}
		\frac{x_{k+1}-x_k}{\sigma} + \partial f(x_{k+1})+A^*y_k & \ni  0\\
		\frac{y_{k+1}-y_k}{\tau} + \partial g(y_{k+1})-Ax_{k+1} & \ni  0,
	\end{aligned}
	\right.   
	\]
	or, equivalently,
	\begin{equation} \label{E:Arrow_Hurwicz}
		\left\{
		\begin{aligned}
			x_{k+1} & = \argmin_{x\in\mathcal X}\left\{f(x)+\frac{1}{2\sigma}\left\|x-\big(x_k-\sigma A^*y_k\big)\right\|^2\right\}\\
			y_{k+1} & =  \argmin_{y\in\mathcal Y}\left\{g(y)+\frac{1}{2\tau}\left\|y-\big(y_k+\tau A^*x_{k+1}\big)\right\|^2\right\}.
		\end{aligned}
		\right.
	\end{equation}
	
	This approach was popularized in image processing by the seminal work of Zhu and Chan \cite{zhu2008efficient}, where a numerical exploration of a linear-quadratic instance of \eqref{P} was presented. Unfortunately, the Arrow-Hurwicz algorithm \eqref{E:Arrow_Hurwicz} is not guaranteed to converge in general, as shown in \cite{he2022convergence}. An extension to the PDHG algorithm, incorporating a momentum substep, was proposed by Chambolle and Pock in \cite{chambolle2011first}, namely
	\begin{equation}\label{CP}
		\left\{
		\begin{aligned}
			x_{k+1} & = \argmin_{x\in\mathcal X}\left\{f(x)+\frac{1}{2\sigma}\left\|x-\big(x_k-\sigma A^*y_k\big)\right\|^2\right\}  \\
			z_{k+1} & = x_{k+1}+\theta(x_{k+1}-x_{k}) \\
			y_{k+1} & =  \argmin_{y\in\mathcal Y}\left\{g(y)+\frac{1}{2\tau}\left\|y-\big(y_k+\tau A^*z_{k+1}\big)\right\|^2\right\},
		\end{aligned}
		\right.
	\end{equation}
	where $\theta\in(0,1]$. With $\theta=1$ and $\tau\sigma\|A\|^2<1$\footnote{These conditions have been extended in \cite{briceno2019projected,banert2026chambolle,upadhyaya2026chambolle}.}, the authors showed that the {\it averaged iterates} converge to a solution of \eqref{PD}, and that the primal-dual gap (a measure of optimality, see below) converges to zero at an {\it ergodic}
	rate of $\mathcal O(1/k)$. This family of algorithms (indexed by the parameter $\theta\in[0,1]$) is now popularly referred to as {\it the} PDHG algorithm. Thanks to its performance, and the existence of theoretical convergence guarantees, the PDHG algorithm has become a well-established method to solve problem \eqref{P}, and a canonical benchmark for competing algorithms. Independently, Condat \cite{condat2013primal} and V\~u \cite{vu2013splitting} proposed a forward-backward splitting extension to tackle problems with a {\it smooth $+$ nonsmooth} structure, and adding a relaxation substep. By expressing it as a particular case of Krasnosel'skii-Mann iterations \cite{Krasnoselskii,Mann}, the algorithm is proved to converge non-ergodically. Convergence rates, however, are still ergodic.
	
	{\bf Non-ergodic convergence rates.} Under (partial) strong convexity, non-ergodic convergence rates have been obtained in \cite{zhu2022new,he2025non}. For general convex functions, non-ergodic convergence rates have been obtained for {\it accelerated} variants of the PDHG algorithm \cite{boct2016inertial, tran2022unified, tran2020adaptive, tran2018smooth, tran2020non}, which include an inertial substep inspired by Nesterov's accelerated gradient method \cite{nesterov1983method}. Recent advances in first-order algorithms have increasingly leveraged augmented Lagrangian-based methods to establish non ergodic convergence guarantees for the iterates (see \cite{tran2019proximal, boct2023fast,tran2018augmented,tran2020non}).

	In the search for non-ergodic convergence rates, an alternative to Nesterov's acceleration, using {\it scalar} preconditioners instead, was proposed in \cite{luo2022primal}. Preconditioning techniques \cite{pock2011diagonal,liu2021acceleration} are known to improve the performance of PDHG methods. These consist in modifying the norms used in the proximal minimization steps in \eqref{CP}, according to the geometry induced by an appropriate elliptic {\it operator} (a positive definite {\it matrix} in the finite dimensional case). The idea is as follows: with $\theta=1$, we can rewrite \eqref{CP} as
	\[\begin{pmatrix}
		I & -\tau A^*\\
		-\sigma A & I 
	\end{pmatrix} \begin{pmatrix}
		\frac{x_{k+1}-x_{k}}{\sigma}\\ \frac{y_{k+1}-y_{k}}{\tau}  
	\end{pmatrix} +\begin{pmatrix}
		\partial f & A^*\\
		-A & \partial g
	\end{pmatrix}  \begin{pmatrix}
		x_{k+1}\\y_{k+1}
	\end{pmatrix}\ni \begin{pmatrix}
		0\\0
	\end{pmatrix},
	\]
	which we can generalize to
	\[\begin{pmatrix}
		\alpha_kI & \beta_k A^*\\
		\gamma_k A & \delta_kI 
	\end{pmatrix} \begin{pmatrix}
		x_{k+1}-x_{k}\\ y_{k+1}-y_{k}  
	\end{pmatrix} +\begin{pmatrix}
		\partial f & A^*\\
		-A & \partial g
	\end{pmatrix}  \begin{pmatrix}
		x_{k+1}\\y_{k+1}
	\end{pmatrix}\ni \begin{pmatrix}
		0\\0
	\end{pmatrix}
	\]
	(we have absorbed the step sizes into the parameter sequences, without loss of generality), and identify as a finite difference discretization of the differential inclusion
	\begin{equation}\label{eq:sist}
		\left\{\begin{array}{rcl}
			\alpha(t)\dot x(t) + \beta(t) A^*\dot y (t)+ \partial f(x(t)) + A^*y(t) & \ni & 0 \\
			\gamma(t)A\dot x(t) + \delta(t)\dot y(t) - Ax(t) + \partial g(y(t)) & \ni & 0.
		\end{array}\right. 
	\end{equation}
	In the context of problem \eqref{COP}, \cite{luo2022primal} propose a system in the form of \eqref{eq:sist} with $\beta(t)\equiv 0$ and $\gamma(t)\equiv-1$. The resulting inclusion, namely
	\[
	\left\{\begin{array}{rcl}
		\alpha(t)\dot x(t)+ \partial f(x(t)) + A^*y(t) & \ni & 0 \\
		-A\dot x(t) + \delta(t)\dot y(t) - Ax(t) + b & \ni & 0,
	\end{array}\right. 
	\]
	then provided inspiration for an algorithm of the form\footnote{We have renamed the parameters for ease of reading: $a_k,b_k,c_k$ depend on $\alpha_k,\delta_k$ and the step size of the discretization.}
	\[
	\left\{\begin{array}{rcl}
		x_{k+1}-x_k+a_k\partial f(x_{k+1}) + a_kA^*y_{k+1} & \ni & 0 \\
		y_{k+1} -y_k - b_k(Ax_k-b) - c_kA(x_{k+1}-x_k) & = & 0.
	\end{array}\right. 
	\]

	In general, this is not directly implementable, but the author of \cite{luo2022primal} proposed a semi-smooth Newton substep to approximate the iterate $(x_{k+1},y_{k+1})$ from $(x_k,y_k)$. These approximation errors are not taken into account in the convergence analysis, though. Now, as general as stated, the inclusion \eqref{eq:sist} was actually introduced in \cite{apidopoulos2026preconditioned}. Using an antisymmetric preconditioner (with $\beta+\gamma\equiv 0$, or even $\beta+\gamma\to 0$), they obtained non-ergodic convergence rates for the solutions of \eqref{eq:sist}. Their results hold without any assumptions on $f$ and $g$, and the choice of $\alpha,\beta,\delta$ does not depend on $f$, $g$ or $A$.

	{\bf Our contribution.} In this paper, we explore the algorithmic consequences of the preceding discussion. Based on antisymmetric preconditioners, we develop a class of algorithms that exploit the {\it smooth $+$ nonsmooth} structure {\it in the saddle point formulation} of the problem. The setting includes, but is not limited to, linearly constrained optimization problems. We then establish non-ergodic convergence rates for the proposed method. Our numerical simulations show the our algorithm is competitive, but also that it can further stabilize the iterates, in comparison with  
	the classical PDHG approach.
	
	The paper is organized as follows: Section \ref{s:main_result} introduces the problem formulation, the proposed preconditioned algorithm, and our main theoretical results. Section \ref{s:proof} details the proof of the central theorem and derives the explicit non-ergodic convergence rates. Section \ref{s:special_cases} discusses relevant particular instances and provides implementation details for the proposed iterations. Finally, Section \ref{s:numerics} validates the algorithm's properties through numerical simulations. By testing our approach on linearly constrained least squares, a linearly constrained $\ell_1+\ell_2$ problem, and image denoising via wavelets, we demonstrate remarkably stable convergence behavior and outperforms standard algorithms from the existing literature.

	\section{Algorithm and main result}\label{s:main_result}
	
	Consider functions $\Phi:\mathcal{X}\to \R$ and $\Gamma:\mathcal{Y}\to \R$ such that $\Phi=f+F$, $\Gamma=g+G$, where $F$ and $G$ are smooth convex functions with parameters $L_F$ and $L_G$, respectively, and $f$ and $g$ are proper, convex and lower semi continuous. Let $A:\mathcal{X} \to \mathcal{Y}$ a continuous linear operator. In this work, we are interested in the numerical approximation of solutions to the \textit{saddle point problem}
	\begin{equation}\label{eq:lag_phi_gamma}
		\min_{x \in \mathcal{X}}\max_{y\in \mathcal{Y}}\lag(x,y) :=\Phi(x)+\inner{Ax}{y}-\Gamma(y).
	\end{equation}
	We denote the set of {\it primal-dual} solutions by $S$. A point $(x^*,y^*)$ belongs to $S$ if, and only if, it satisfies the first order optimality condition:
	\begin{equation} \label{eq:optimality_condition}
		-A^*y^*\in\partial\Phi(x^*)=\nabla F(x^*)+\partial f(x^*)\qqbox{and}Ax^*\in\partial\Gamma(y^*)=\nabla G(y^*)+\partial g(y^*).
	\end{equation}
	
	\subsection{Primal-dual iterations}
	
	Following \cite{apidopoulos2026preconditioned}, we consider
	\begin{equation}\label{eq:sist_extended}
		\left\{\begin{array}{rcl}
			\alpha(t)\dot x(t) + \beta(t) A^*\dot y (t)+ \partial f(x(t)) +\nabla F(x(t))+ A^*y(t) & \ni & 0 \\
			-\beta(t)A\dot x(t) + \delta(t)\dot y(t) - Ax(t) + \partial g(y(t)) +\nabla G(y(t)) & \ni & 0,
		\end{array}\right. 
	\end{equation}
	a semi-implicit discretization of which gives 
	\begin{equation}\label{eq:alg1_composite}
		\left\{\begin{array}{rcl}
			\ds \alpha_k(x_{k+1}-x_k) +\beta_kA^* (y_{k+1}-y_k)+ \partial f(x_{k+1})+\nabla F(x_k)+ A^*y_{k+1} &  \ni & 0 \medskip \\
			\ds -\beta_kA(x_{k+1}-x_k) +\delta_k (y_{k+1}-y_k) - Ax_{k+1} +  \partial g(y_{k+1})+\nabla G(y_{k}) & \ni & 0.
		\end{array}\right.
	\end{equation}

	In matrix form, this can be expressed as
	\[
	\begin{pmatrix}
		\alpha_kI + \partial f & 0\\
		0 & \delta_kI + \partial g
	\end{pmatrix}  \begin{pmatrix}
		x_{k+1}\\y_{k+1}
	\end{pmatrix} +(\beta_k+1)\begin{pmatrix}
		0 & A^*\\
		- A & 0
	\end{pmatrix}  \begin{pmatrix}
		x_{k+1}\\y_{k+1}
	\end{pmatrix}\ni \begin{pmatrix}
		\tilde x_k \\ \tilde y_k
	\end{pmatrix},
		\]
		where
		\[
		\begin{pmatrix}
			\tilde x_{k} \\ \tilde y_{k}
		\end{pmatrix}
		= \begin{pmatrix}
			\alpha_kI - \nabla F & \beta_k A^*\\
			-\beta_k A & \delta_kI  - \nabla G
		\end{pmatrix} \begin{pmatrix}
			x_{k} \\ y_{k}
		\end{pmatrix} 
		\]
		At each iteration, one has to find a zero of a {\it strongly monotone operator} with a {\it diagonal $+$ skew} structure. In all generality, this iteration subproblem cannot be solved explicitly, but there are efficient algorithms to solve it approximately (especially, see \cite{briceno2011monotone+}, but also \cite{combettes2009iterative,aragon2019computing,dao2020computing,adly2019decomposition}). Therefore, we consider an inexact version of \eqref{eq:alg1_composite}, namely
		\begin{equation}\label{eq:alg1_composite_inexact}
			\left\{\begin{array}{rcl}
				\ds \alpha_k(x_{k+1}-x_k) +\beta_kA^* (y_{k+1}-y_k)+ \partial f(x_{k+1})+\nabla F(x_k)+ A^*y_{k+1} &  \ni & \varepsilon_{k+1} \medskip \\
				\ds -\beta_kA(x_{k+1}-x_k) +\delta_k (y_{k+1}-y_k) - Ax_{k+1} +  \partial g(y_{k+1})+\nabla G(y_{k}) & \ni & \epsilon_{k+1},
			\end{array}\right.
		\end{equation}
		where the vectors $\varepsilon_{k+1}\in\cali X$ and $\epsilon_{k+1}\in\cali Y$ are to be interpreted as approximation errors in the (possibly) inexact implementation of the algorithm. There are, however, relevant cases where the structure of the iterations becomes split. Some of these are discussed in Section \ref{s:special_cases}. For the convergence results, we keep the general form \eqref{eq:alg1_composite_inexact} for compactness.
		
		\begin{remark}\label{r:iterations_shift}
			Setting $\tilde\beta_k:=\beta_k+1$, we can equivalently write \eqref{eq:alg1_composite_inexact} as 
			\[
			\left\{\begin{array}{rcl}
				\ds \alpha_k(x_{k+1}-x_k) +\tilde\beta_kA^* (y_{k+1}-y_k)+ \partial f(x_{k+1})+\nabla F(x_k)+ A^*y_k &  \ni &  \varepsilon_{k+1} \medskip \\
				\ds -\tilde\beta_kA(x_{k+1}-x_k) +\delta_k (y_{k+1}-y_k) - Ax_k +  \partial g(y_{k+1})+\nabla G(y_{k}) & \ni & \epsilon_{k+1}.
			\end{array}\right.
			\]
		\end{remark}
		
		\begin{remark}
			One can also consider variations of \eqref{eq:alg1_composite_inexact} with an index mismatch, such as
			\begin{equation} \label{eq:alg_1_a}
				\left\{\begin{array}{rcl}
					\ds \alpha_k(x_{k+1}-x_k) +\beta_kA^* (y_{k+1}-y_k)+ \partial f(x_{k+1})+\nabla F(x_k)+ A^*y_{k+1} &  \ni & \varepsilon_{k+1} \medskip \\
					\ds -\beta_kA(x_{k+1}-x_k) +\delta_k (y_{k+1}-y_k) - Ax_k +  \partial g(y_{k+1})+\nabla G(y_{k}) & \ni & \epsilon_{k+1},
				\end{array}\right.    
			\end{equation}
			or 
			\begin{equation} \label{eq:alg_1_b}
				\left\{\begin{array}{rcl}
					\ds \alpha_k(x_{k+1}-x_k) +\beta_kA^* (y_{k+1}-y_k)+ \partial f(x_{k+1})+\nabla F(x_k)+ A^*y_k &  \ni & \varepsilon_{k+1} \medskip \\
					\ds -\beta_kA(x_{k+1}-x_k) +\delta_k (y_{k+1}-y_k) - Ax_{k+1} +  \partial g(y_{k+1})+\nabla G(y_{k}) & \ni & \epsilon_{k+1}.
				\end{array}\right.    
			\end{equation}
			They exhibit similar numerical behavior, and the convergence analysis is analogous, so we will omit it.
		\end{remark}
		
		\subsection{Non-ergodic convergence rate}
		
		Given $(x^*,y^*) \in S$, we define the \textit{primal-dual 
			gap} as 
		\begin{align*}
			\Delta_k  :&= \lag\paren{x_{k},y^*} -  \lag\paren{x^*,y_k} \\
			& = \big[\Phi(x_{k})+\inner{Ax_k}{y^*}-\Gamma(y^*)\big]-\big[\Phi(x^*)+\inner{Ax^*}{y_k}-\Gamma(y_k)\big] \\
			& =\big[\Phi(x_k)-\Phi(x^*)+\inner{A^*y^*}{x_k-x^*}\big]+\big[\Gamma(y_k)-\Gamma(y^*)-\inner{Ax^*}{y_k-y^*}\big]
		\end{align*}
		In view of the convexity of $\Phi$ and $\Gamma$ and the optimality conditions \eqref{eq:optimality_condition}, we have
		$$\Delta_k\ge \Phi(x_k)-\Phi(x^*)+\inner{A^*y^*}{x_k-x^*}\ge 0,$$
		$$\Delta_k\ge \Gamma(y_k)-\Gamma(y^*)-\inner{Ax^*}{y_k-y^*}\ge 0.$$
		In particular,
		
		\begin{equation}\label{eq:decay_Phi_Gamma}
			\begin{array}{rl}
				\Phi(x_k)-\Phi(x^*)+\inner{A^*y^*}{x_k-x^*}&=\mathcal O(\Delta_k), \medskip \\
				\Gamma(y_k)-\Gamma(y^*)-\inner{Ax^*}{y_k-y^*}&=\mathcal O(\Delta_k). \\
			\end{array}
		\end{equation}
		By lower-semicontinuity, if $\Delta_k\to 0$ as $k\to\infty$, then every weak subsequential limit point of $(x_k,y_k)$ is a primal-dual solution of \eqref{eq:lag_phi_gamma}.

		Our analysis relies on the energy function
		\[E_{k}:=\Delta_{k}+X_{k}(x^*)+Y_{k}(y^*),\]
		where
		\[X_{k}(x^*)=\frac{\tau_{k}}{2}\norm{x_k-x^*}^2 \quad\text{and}\quad Y_{k}(y^*)=\frac{\sigma_k}{2}\norm{y_k-y^*}^2,\]
		with $\tau_{k}:=\frac{\alpha_{k}}{\beta_{k}}$ and $\sigma_{k}:=\frac{\delta_{k}}{\beta_{k}}$. Note that $E_k$ is the sum of three nonnegative terms.
		
		We shall make the following assumption on the parameters:
		
		\begin{assumption} \label{a:a_composite_inexact} 
			The sequences $(\alpha_k)$, $(\beta_k)$ and $(\delta_k)$ are positive and there is $\theta\in(0,1)$ such that
			\[(1-\theta)\alpha_k \ge \dfrac{(\beta_{k}+1)L_F}{2\beta_k+1}, \quad (1-\theta)\delta_k \ge \dfrac{(\beta_{k}+1)L_G}{2\beta_k+1}, \quad \dfrac{\beta_{k+1}}{\alpha_{k+1}} \geq \dfrac{1+\beta_k}{\alpha_k}, \quad \dfrac{\beta_{k+1}}{\delta_{k+1}} \geq \dfrac{1+\beta_k}{\delta_k}.\]    
		\end{assumption}
		
		The rationale for Assumption \ref{a:a_composite_inexact} will become transparent later, in particular thanks to Lemma \ref{L:composite_inexact}. On the other hand, we can use $\theta=0$ in the error-free case.

		\begin{remark} \label{re:alg1}
			The first two inequalities in Assumption \ref{a:a_composite_inexact} imply that $\alpha_k(2\beta_k+1)-(\beta_k+1)L_F\ge \theta\alpha_k(2\beta_k+1)$ and $\delta_k(2\beta_k+1)-(\beta_k+1)L_G\ge \theta\delta_k(2\beta_k+1)$. In terms of $(\tau_k)$ and $(\sigma_k)$, the last two inequalities read $\left(1+\frac{1}{\beta_k}\right)\tau_{k+1} \le \tau_k$ and $\left(1+\frac{1}{\beta_k}\right)\sigma_{k+1} \le \sigma_k$, respectively.
		\end{remark}
		
		Our main theoretical result is the following:
		
		\begin{theorem} \label{T:implicit_inexact}
			Let $(x^*,y^*) \in S$, and let $(x_k,y_k)$ be generated by Algorithm \eqref{eq:alg1_composite_inexact} with parameters satisfying Assumption \ref{a:a_composite_inexact}. Then, 
			\begin{equation} \label{E:rate_inexact}
				\sqrt{E_{k+1}}\le \frac{1}{B_k}\left(\sqrt{E_0} + \left(1+\sqrt{1+\frac{1}{2\theta}}\right)\sum_{j=0}^k\beta_jB_jR_j\right),
			\end{equation}
			where
			$$B_j:=\prod_{i=0}^j\sqrt{1+\frac{1}{\beta_i}}.$$
			In particular, we have the following:
			\begin{itemize}
				\item [i)] If $(1/\beta_k)\notin\ell^1$ and $(\beta_kR_k)\in\ell^1$, then $\lim_{k\to\infty}\Delta_k=\lim_{k\to\infty}E_k=0$.
				\item [ii)] If $(\beta_kB_kR_k)\in\ell^1$, then $\Delta_k\le E_k=\mathcal O(1/B_k^2)$.
			\end{itemize}
			In any case, $\Phi(x_k) - \Phi(x^*) + \left\langle Ax_k,y^*\right\rangle=\mathcal O(\Delta_k)$ and $\Gamma(y_k) - \Gamma(y^*) - \left\langle Ax^*,y_k\right\rangle=\mathcal O(\Delta_k)$, and every weak subsequential limit point of $(x_k,y_k)$ is a primal-dual solution of \eqref{eq:lag_phi_gamma}.
		\end{theorem}
		
		The proof is postponed briefly, to Section \ref{s:proof}.

		\subsection{The role of the parameters}
		
		Let us now discuss some concrete instances of Theorem \ref{T:implicit_inexact}. In the following examples, we assume $\alpha_k=\delta_k$ and $\|\varepsilon_k\|=\|\epsilon_k\|=:e_k$ in order to simplify the exposition. 
		
		\begin{example}
			Suppose first that $\beta_k\equiv \beta_0$. In the {\it error-free} case, where $e_k\equiv 0$, we have
			\begin{equation} \label{E:linear_rate_betabounded}
				\Delta_k\le E_0\left(\frac{\beta_0}{1+\beta_0}\right)^k.
			\end{equation}
			In other words, $\Delta_k$ converges linearly to zero. If the errors are such that
			$$\sum_{k=1}^\infty \left(\frac{1+\beta_0}{\beta_0}\right)^ke_k<\infty,$$
			the linear convergence of $\Delta_k$ to 0 given by \eqref{E:linear_rate_betabounded} is preserved. On the other hand, if
			$$\beta_kR_k=\sqrt{\beta_0}\frac{e_k}{\sqrt{\alpha_k}}\ge \sqrt{\frac{\beta_0}{\alpha_0}}\left(\frac{1+\beta_0}{\beta_0}\right)^{\frac{k}{2}}e_k$$
			is summable, the minimization property is preserved, albeit without rates. 
		\end{example}
		
		\begin{example}\label{ex:parameters_errors}
			For $\alpha\equiv\alpha_0$, equality in Assumption \ref{a:a_composite_inexact} gives $\beta_k=\beta_0+k$, and so
			$$\beta_kR_k=\frac{\sqrt{\beta_0+k+1}e_{k+1}}{\sqrt{\alpha_0}}.$$
			It follows that $\Delta_k\to 0$ under the much milder assumption that $\ds\sum_{k=1}^\infty \sqrt{k}e_k<\infty$. In the error-free case, or if $\ds\sum_{k=1}^\infty ke_k<\infty$, we get $\Delta_k=\mathcal O\left(\frac{1}{B_k^2}\right)=\mathcal O\left(\frac{1}{k}\right)$. 
		\end{example}
		
		\begin{example}
			For $\alpha_k=\alpha_0+k$ (and also with the equality in Assumption \ref{a:a_composite_inexact}) we have 
			$$\beta_{k+1}=\left(\frac{\alpha_0+k+1}{\alpha_0+k}\right)(\beta_k+1).$$
			This leads to
			$$\beta_{k+1}=(\alpha_0+k+1)\left[\frac{\beta_0}{\alpha_0}+\sum_{j=0}^k\frac{1}{\alpha_0+j}\right]\sim (\alpha_0+k+1)\ln(\alpha_0+k+1)\qbox{and}\frac{\beta_{k+1}}{\alpha_{k+1}}\sim\ln(\alpha_0+k+1).$$
			As a consequence, $\Delta_k\to 0$ if $\ds\sum_{j=1}^\infty \sqrt{\ln(j)}e_j<\infty$. In the error-free case, or if $\ds\sum_{j=1}^\infty \ln(j)e_j<\infty$, we have $\Delta_k=\mathcal O\left(\frac{1}{\sqrt{\ln(k)}}\right)$.
		\end{example}

		\section{Proof of Theorem \ref{T:implicit_inexact}} \label{s:proof}
		
		Fix any $(x^*,y^*) \in S$. Recall that
		\[E_{k}=\Delta_{k}+X_{k}(x^*)+Y_{k}(y^*),\]
		where
		\begin{align*}
			\Delta_k & = \lag\paren{x_{k},y^*} -  \lag\paren{x^*,y_k} \\
			& = \big[\Phi(x_{k})+\inner{Ax_k}{y^*}-\Gamma(y^*)\big]-\big[\Phi(x^*)+\inner{Ax^*}{y_k}-\Gamma(y_k)\big] \\
			& = \big[\Phi(x_k)-\Phi(x^*)\big]+\big[\Gamma(y_k)-\Gamma(y^*)\big]+\big[\inner{Ax_k}{y^*}-\inner{Ax^*}{y_k}\big],
		\end{align*}
		and
		\[X_{k}(x^*)=\frac{\tau_{k}}{2}\norm{x_k-x^*}^2 \quad\text{and}\quad Y_{k}(y^*)=\frac{\sigma_k}{2}\norm{y_k-y^*}^2,\]
		with $\tau_{k}=\frac{\alpha_{k}}{\beta_{k}}$ and $\sigma_{k}=\frac{\delta_{k}}{\beta_{k}}$.
		
		We have the following energy decrease estimation, from which the motivation for Assumption \ref{a:a_composite_inexact} becomes evident:
		
		\begin{lemma} \label{L:composite_inexact}
			Let $(x^*,y^*) \in S$, and let $(x_k,y_k)$ be generated by Algorithm \eqref{eq:alg1_composite_inexact}. Then,
			\begin{align*}
				\left(1+\frac{1}{\beta_k}\right)&E_{k+1}-E_k \\
				& \le  \frac{1}{2}\left[\left(1+\frac{1}{\beta_k}\right)\tau_{k+1}-\tau_k\right]\|x_{k+1}-x^*\|^2 + \frac{1}{2}\left[\left(1+\frac{1}{\beta_k}\right)\sigma_{k+1}-\sigma_k\right]|y_{k+1}-y^*\|^2 \\
				&\quad + \left[\frac{(\beta_k+1)L_{F}-(2\beta_k+1)\alpha_k}{2\beta_k}\right]\|x_{k+1}-x_k\|^2 \\
				&\quad + \left[\frac{(\beta_k+1)L_{G}-(2\beta_k+1)\delta_k}{2\beta_k}\right]|y_{k+1}-y_k\|^2\\
				&\quad + \inner{\varepsilon_{k+1}}{x_{k+1}-x_k} + \frac{\|\varepsilon_{k+1}\|}{\beta_k}\|x_{k+1}-x^*\| +  \inner{\epsilon_{k+1}}{y_{k+1}-y_k} +  \frac{\|\epsilon_{k+1}\|}{\beta_k}\|y_{k+1}-y^*\|.
			\end{align*}
		\end{lemma}
		
		\begin{proof}
			We begin by writing
			\begin{align*}
				\paren{1+\frac{1}{\beta_k}}&\Delta_{k+1}-\Delta_k \\
				& =  \big(\Phi(x_{k+1})-\Phi(x_k)\big) + \frac{1}{\beta_k}\big(\Phi(x_{k+1})-\Phi(x^*)\big) \\
				&\quad + \big(\Gamma(y_{k+1})-\Gamma(y_k)\big) + \frac{1}{\beta_k}\big(\Gamma(y_{k+1})-\Gamma(y^*)\big)\\
				&\quad + \inner{A(x_{k+1}-x_k)}{y^*}-\inner{Ax^*}{y_{k+1}-y_k}+\frac{1}{\beta_k}\big(\inner{Ax_{k+1}}{y^*}-\inner{Ax^*}{y_{k+1}}\big) \\
				& = \Big[\Phi(x_{k+1})-\Phi(x_k)\Big] + \Big[\Gamma(y_{k+1})-\Gamma(y_k)\Big] + \Big[\inner{A(x_{k+1}-x_k)}{y^*}-\inner{Ax^*}{y_{k+1}-y_k}\Big] \\
				& \quad + \frac{1}{\beta_k}\bigg[\big(\Phi(x_{k+1})-\Phi(x^*)\big) + \big(\Gamma(y_{k+1})-\Gamma(y^*)\big) + \big(\inner{Ax_{k+1}}{y^*}-\inner{Ax^*}{y_{k+1}}\big)\bigg]. 
			\end{align*}
			Set 
			\[p_{k+1} := A^*\paren{-\beta_k\paren{y_{k+1}-y_k}-y_{k+1}} - \alpha_k\paren{x_{k+1}-x_{k}},\]
			so that $p_{k+1} +\varepsilon_{k+1}-\nabla F(x_k)\in \partial f\paren{x_{k+1}}$. By convexity, we have
			$$f(x_{k+1})\le f(x_k)+\inner{p_{k+1} +\varepsilon_{k+1}-\nabla F(x_k)}{x_{k+1}-x_k}.$$
			But we also have
			$$F(x_{k+1}) \le F(x_k) + \inner{\nabla F(x_k)}{x_{k+1}-x_k}+\frac{L_{F}}{2}\|x_{k+1}-x_k\|^2,$$
			from which it follows that
			$$\Phi(x_{k+1})\le \Phi(x_k)+\inner{p_{k+1} +\varepsilon_{k+1}}{x_{k+1}-x_k}+\frac{L_{F}}{2}\|x_{k+1}-x_k\|^2.$$
			On the other hand, we have
			$$f(x_{k+1})\le f(x^*)+\inner{p_{k+1} +\varepsilon_{k+1}-\nabla F(x_k)}{x_{k+1}-x^*}$$
			and
			$$F(x_k)\le F(x^*)+\inner{\nabla F(x_k)}{x_k-x^*}.$$
			Combining these inequalities, we get
			$$\Phi(x_{k+1})\le \Phi(x^*)+\inner{p_{k+1} +\varepsilon_{k+1}}{x_{k+1}-x^*}+\frac{L_{F}}{2}\|x_{k+1}-x_k\|^2.$$
			Similarly, setting 
			\[q_{k+1} := \beta_{k}A(x_{k+1}-x_{k})-\delta_{k}(y_{k+1}-y_{k})+Ax_{k+1},\]
			we have $q_{k+1} +\epsilon_{k+1}-\nabla G(y_{k})\in \partial g^*\paren{y_{k+1}}$, whence
			$$\Gamma(y_{k+1})\le \Gamma(y_k)+\inner{q_{k+1} +\epsilon_{k+1}}{y_{k+1}-y_k}+\frac{L_{G}}{2}\|y_{k+1}-y_k\|^2$$
			and
			$$\Gamma(y_{k+1})\le \Gamma(y^*)+\inner{q_{k+1} +\epsilon_{k+1}}{y_{k+1}-y^*}+\frac{L_{G}}{2}\|y_{k+1}-y_k\|^2.$$
			The discussion above gives
			\begin{align*}
				\paren{1+\frac{1}{\beta_k}}&\Delta_{k+1}-\Delta_k\\
				=& \Big[\Phi(x_{k+1})-\Phi(x_k)\Big] + \Big[\Gamma(y_{k+1})-\Gamma(y_k)\Big] + \Big[\inner{A(x_{k+1}-x_k)}{y^*}-\inner{Ax^*}{y_{k+1}-y_k}\Big] \\
				& + \frac{1}{\beta_k}\bigg[\big(\Phi(x_{k+1})-\Phi(x^*)\big) + \big(\Gamma(y_{k+1})-\Gamma(y^*)\big) + \big(\inner{Ax_{k+1}}{y^*}-\inner{Ax^*}{y_{k+1}}\big)\bigg]  \\
				\le&  \inner{p_{k+1} +\varepsilon_{k+1}}{x_{k+1}-x_k}+\frac{L_{F}}{2}\|x_{k+1}-x_k\|^2  + \inner{q_{k+1} +\epsilon_{k+1}}{y_{k+1}-y_k}+\frac{L_{G}}{2}\|y_{k+1}-y_k\|^2 \\
				&  + \Big[\inner{A(x_{k+1}-x_k)}{y^*}-\inner{Ax^*}{y_{k+1}-y_k}\Big] \\
				&  + \frac{1}{\beta_k}\bigg[\inner{p_{k+1} +\varepsilon_{k+1}}{x_{k+1}-x^*}+\frac{L_{F}}{2}\|x_{k+1}-x_k\|^2\bigg]  \\
				&  + \frac{1}{\beta_k}\bigg[\inner{q_{k+1} +\epsilon_{k+1}}{y_{k+1}-y^*}+\frac{L_{G}}{2}\|y_{k+1}-y_k\|^2\bigg]  \\
				& + \frac{1}{\beta_k}\bigg[\inner{Ax_{k+1}}{y^*}-\inner{Ax^*}{y_{k+1}}\bigg] \\
				=&  \inner{p_{k+1} + A^*y^*}{x_{k+1}-x_k}  + \inner{q_{k+1}-Ax^*}{y_{k+1}-y_k}  \\
				& \quad  + \frac{1}{\beta_k}\bigg[ \inner{p_{k+1} }{x_{k+1}-x^*}  + \inner{q_{k+1}}{y_{k+1}-y^*} + \inner{Ax_{k+1}}{y^*}-\inner{Ax^*}{y_{k+1}}\bigg] \\
				& \quad + \frac{(\beta_k+1)L_{F}}{2\beta_k}\|x_{k+1}-x_k\|^2 + \frac{(\beta_k+1)L_{G}}{2\beta_k}\|y_{k+1}-y_k\|^2\\
				& \quad + \inner{\varepsilon_{k+1}}{x_{k+1}-x_k+\frac{1}{\beta_k}(x_{k+1}-x^*)} +  \inner{\epsilon_{k+1}}{y_{k+1}-y_k+\frac{1}{\beta_k}(y_{k+1}-y^*)}  \\
				=& \inner{p_{k+1} + A^*y^*}{x_{k+1}-x_k}  + \inner{q_{k+1}-Ax^*}{y_{k+1}-y_k}  \\
				& \quad  + \frac{1}{\beta_k}\Big[ \inner{p_{k+1} +A^*y_{k+1}}{x_{k+1}-x^*}  + \inner{q_{k+1}-Ax_{k+1}}{y_{k+1}-y^*}\Big] \\
				& \quad + \frac{(\beta_k+1)L_{F}}{2\beta_k}\|x_{k+1}-x_k\|^2 + \frac{(\beta_k+1)L_{G}}{2\beta_k}\|y_{k+1}-y_k\|^2\\
				& \quad + \inner{\varepsilon_{k+1}}{x_{k+1}-x_k+\frac{1}{\beta_k}(x_{k+1}-x^*)} +  \inner{\epsilon_{k+1}}{y_{k+1}-y_k+\frac{1}{\beta_k}(y_{k+1}-y^*)}.
			\end{align*}
			We focus our attention on the first line on the right-hand side. From the definitions 
			\begin{align*}
				p_{k+1} & = A^*\paren{-\beta_k\paren{y_{k+1}-y_k}-y_{k+1}} - \alpha_k\paren{x_{k+1}-x_{k}} \\
				q_{k+1} & = \beta_{k}A(x_{k+1}-x_{k})-\delta_{k}(y_{k+1}-y_{k})+Ax_{k+1},
			\end{align*}
			it follows that
			\begin{align*}
				&\inner{p_{k+1} + A^*y^*}{x_{k+1}-x_k}  + \inner{q_{k+1}-Ax^*}{y_{k+1}-y_k} \\
				=&  -\inner{\beta_k\paren{y_{k+1}-y_k}+y_{k+1}-y^*}{A(x_{k+1}-x_k)} +\inner{\beta_{k}(x_{k+1}-x_{k})+x_{k+1}-x^*}{A^*(y_{k+1}-y_k)}
				\\
				& -\alpha_k\|x_{k+1}-x_k\|^2 -\delta_k\|y_{k+1}-y_k\|^2 \\
				=&  -\inner{y_{k+1}-y^*}{A(x_{k+1}-x_k)}  +\inner{x_{k+1}-x^*}{A^*(y_{k+1}-y_k)}
				\\
				& -\alpha_k\|x_{k+1}-x_k\|^2 -\delta_k\|y_{k+1}-y_k\|^2.
			\end{align*}
			Similarly, for the second line, we do
			\begin{align*}
				&\inner{p_{k+1} + A^*y_{k+1}}{x_{k+1}-x^*}  + \inner{q_{k+1}-Ax_{k+1}}{y_{k+1}-y^*} \\
				=& -\beta_k\inner{y_{k+1}-y_k}{A(x_{k+1}-x^*)} -\alpha_k\inner{x_{k+1}-x_k}{x_{k+1}-x^*}\\
				&+\beta_{k}\inner{x_{k+1}-x_{k}}{A^*(y_{k+1}-y^*)} -\delta_k\inner{y_{k+1}-y_k}{y_{k+1}-y^*}.
			\end{align*}
			Multiplying the second equality by $1/\beta_k$, and adding it to the first one, we get
			\begin{align*}
				\left(1+\frac{1}{\beta_k}\right)&\Delta_{k+1}-\Delta_k\\
				&  \le -\frac{\alpha_k}{\beta_k}\inner{x_{k+1}-x_k}{x_{k+1}-x^*} -\frac{\delta_k}{\beta_k}\inner{y_{k+1}-y_k}{y_{k+1}-y^*} \\
				& \quad + \left[\frac{(\beta_k+1)L_{F}}{2\beta_k}-\alpha_k\right]\|x_{k+1}-x_k\|^2 + \left[\frac{(\beta_k+1)L_{G}}{2\beta_k}-\delta_k\right]|y_{k+1}-y_k\|^2\\
				& \quad + \inner{\varepsilon_{k+1}}{x_{k+1}-x_k+\frac{1}{\beta_k}(x_{k+1}-x^*)} +  \inner{\epsilon_{k+1}}{y_{k+1}-y_k+\frac{1}{\beta_k}(y_{k+1}-y^*)}.
			\end{align*}
			We now transform the first line on the right-hand side into a sum of squares, obtaining
			\begin{align*}
				\left(1+\frac{1}{\beta_k}\right)&\Delta_{k+1}-\Delta_k\\
				& \le \frac{\alpha_k}{2\beta_k}\Big[\|x_k-x^*\|^2-\|x_{k+1}-x^*\|^2-\|x_{k+1}-x_k\|^2\Big] \\
				& \quad +\frac{\delta_k}{2\beta_k}\Big[\|y_k-x^*\|^2-\|y_{k+1}-x^*\|^2-\|y_{k+1}-y_k\|^2\Big] \\
				& \quad + \left[\frac{(\beta_k+1)L_{F}}{2\beta_k}-\alpha_k\right]\|x_{k+1}-x_k\|^2 + \left[\frac{(\beta_k+1)L_{G}}{2\beta_k}-\delta_k\right]|y_{k+1}-y_k\|^2\\
				& \quad + \inner{\varepsilon_{k+1}}{x_{k+1}-x_k+\frac{1}{\beta_k}(x_{k+1}-x^*)} +  \inner{\epsilon_{k+1}}{y_{k+1}-y_k+\frac{1}{\beta_k}(y_{k+1}-y^*)} \\
				& = \frac{\tau_k}{2}\Big[\|x_k-x^*\|^2-\|x_{k+1}-x^*\|^2\Big] + \frac{\sigma_k}{2}\Big[\|y_k-x^*\|^2-\|y_{k+1}-x^*\|^2\Big] \\
				& \quad + \left[\frac{(\beta_k+1)L_{F}-\alpha_k}{2\beta_k}-\alpha_k\right]\|x_{k+1}-x_k\|^2 + \left[\frac{(\beta_k+1)L_{G}-\delta_k}{2\beta_k}-\delta_k\right]|y_{k+1}-y_k\|^2\\
				& \quad + \inner{\varepsilon_{k+1}}{x_{k+1}-x_k+\frac{1}{\beta_k}(x_{k+1}-x^*)} +  \inner{\epsilon_{k+1}}{y_{k+1}-y_k+\frac{1}{\beta_k}(y_{k+1}-y^*)}
			\end{align*}
			Finally, we have
			\begin{align*}
				\left(1+\frac{1}{\beta_k}\right)X_{k+1}(x^*)-X_k(x^*) 
				& = \left(1+\frac{1}{\beta_k}\right)\frac{\tau_{k+1}}{2}\norm{x_{k+1}-x^*}^2-\frac{\tau_{k}}{2}\norm{x_k-x^*}^2 \\
				\left(1+\frac{1}{\beta_k}\right)Y_{k+1}(y^*)-Y_k(y^*) 
				& = \left(1+\frac{1}{\beta_k}\right)\frac{\sigma_{k+1}}{2}\norm{y_{k+1}-y^*}^2-\frac{\sigma_{k}}{2}\norm{y_k-y^*}^2,
			\end{align*}
			which, combined with the inequality above, give the desired result.
		\end{proof}
		
		As a consequence, we obtain:
		
		\begin{proposition}\label{prop:E_decay_composite_inexact}
			Let $(x^*,y^*) \in S$, and let $(x_k,y_k)$ be generated by Algorithm \eqref{eq:alg1_composite_inexact} with the parameters satisfying Assumption \ref{a:a_composite_inexact}. Then, 
			\begin{equation} \label{E:Quadratic_inequality_composite_inexact}
				\left(1+\frac{1}{\beta_k}\right)E_{k+1}-E_k \le 2R_k\sqrt{E_{k+1}}+\frac{\beta_k}{2\theta}R_k^2,
			\end{equation}
			where we have written
			\begin{equation} \label{E:R_k_composite_inexact}
				R_k:=\frac{\sqrt{\beta_{k+1}}}{\beta_k}\max\left\{\frac{\|\varepsilon_{k+1}\|}{\sqrt{\alpha_{k+1}}}, \frac{\|\epsilon_{k+1}\|}{\sqrt{\delta_{k+1}}}\right\}.
			\end{equation}
		\end{proposition}

		\begin{proof} 
			Under Assumption \ref{a:a_composite_inexact}, Lemma \ref{L:composite_inexact} gives
			\begin{align*}
				&\left(1+\frac{1}{\beta_k}\right)E_{k+1}-E_k \\
				\le&   
				- \frac{\theta\alpha_k(2\beta_k+1)}{2\beta_k}\|x_{k+1}-x_k\|^2 
				- \frac{\theta\delta_k(2\beta_k+1)}{2\beta_k}\|y_{k+1}-y_k\|^2\\
				& + \inner{\varepsilon_{k+1}}{x_{k+1}-x_k} + \frac{\|\varepsilon_{k+1}\|}{\beta_k}\|x_{k+1}-x^*\| +  \inner{\epsilon_{k+1}}{y_{k+1}-y_k} +  \frac{\|\epsilon_{k+1}\|}{\beta_k}\|y_{k+1}-y^*\|.
			\end{align*}
			We use Young's inequality on
			$\inner{\varepsilon_{k+1}}{x_{k+1}-x_k}$ and $ \inner{\epsilon_{k+1}}{y_{k+1}-y_k}$, to get
			\begin{align*}
				\left(1+\frac{1}{\beta_k}\right)E_{k+1}-E_k & \le  \frac{\|\varepsilon_{k+1}\|}{\beta_k}\|x_{k+1}-x^*\| +  \frac{\|\epsilon_{k+1}\|}{\beta_k}\|y_{k+1}-y^*\| \\
				& \quad  + \frac{\beta_k\|\varepsilon_{k+1}\|^2}{2\theta\alpha_k(2\beta_k+1)} + \frac{\beta_k\|\epsilon_{k+1}\|^2}{2\theta\delta_k(2\beta_k+1)} \\
				& \le \sqrt{2}R_k\Big[\sqrt{X_{k+1}(x^*)}+\sqrt{Y_{k+1}(y^*)}\Big] \\
				& \quad 
				+ \frac{\beta_k^3R_k^2\alpha_{k+1}}{2\theta\alpha_k\beta_{k+1}(2\beta_k+1)} + \frac{\beta_k^3R_k^2\delta_{k+1}}{2\theta\delta_k\beta_{k+1}(2\beta_k+1)} \\
				& \le 2R_k\sqrt{E_{k+1}}
				+ \frac{\beta_k^3R_k^2}{2\theta(\beta_k+1)(2\beta_k+1)} + \frac{\beta_k^3R_k^2}{2\theta(\beta_k+1)(2\beta_k+1)},
			\end{align*}
			where we have used Assumption \ref{a:a_composite_inexact} again, and the definition of $R_k$, given in \eqref{E:R_k_composite_inexact}. This clearly gives \eqref{E:Quadratic_inequality_composite_inexact}.
		\end{proof}
		
		We are now in a position to complete the proof of Theorem \ref{T:implicit_inexact}:
		
		\begin{proof}
			Inequality \eqref{E:Quadratic_inequality_composite_inexact}, which we rewrite as
			$$\left(1+\frac{1}{\beta_k}\right)E_{k+1}-2R_k\sqrt{E_{k+1}}-\left[E_k + \frac{\beta_k}{2\theta}R_k^2\right] \le 0,$$
			is quadratic with respect to $\sqrt{E_{k+1}}\ge 0$, and implies that 
			$$\sqrt{E_{k+1}}
			\le \frac{R_k+\sqrt{R_k^2+\left(1+\frac{1}{\beta_k}\right)\left(E_k+\frac{\beta_k}{2\theta}R_k^2\right)}}{\left(1+\frac{1}{\beta_k}\right)} 
			= \frac{R_k+\sqrt{\left(1+\frac{\beta_k+1}{2\theta}\right)R_k^2+\left(1+\frac{1}{\beta_k}\right)E_k}}{\left(1+\frac{1}{\beta_k}\right)}.$$
			Since $\sqrt{a+b}\le \sqrt{a}+\sqrt{b}$ for $a,b\ge 0$, this gives
			\begin{align*}
				\sqrt{E_{k+1}}
				& \le \sqrt{\frac{\beta_k}{1+\beta_k}}\sqrt{E_k}+\frac{\beta_k}{1+\beta_k}\left(1+\sqrt{1+\frac{\beta_k+1}{2\theta}}\right)R_k \\
				& \le \sqrt{\frac{\beta_k}{1+\beta_k}}\sqrt{E_k} + \left(1+\sqrt{1+\frac{1}{2\theta}}\right)\beta_kR_k,   
			\end{align*}
			where we have used that
			$$\frac{1}{1+\beta_k}\left(1+\sqrt{1+\frac{\beta_k+1}{2\theta}}\right)\le 1+\sqrt{1+\frac{1}{2\theta}},$$
			which is true for every $\beta_k\ge 0$. We then iterate this inequality to obtain \eqref{E:rate_inexact}. For i), the sequence $(B_j)$ is increasing and, if $(1/\beta_k)\notin\ell^1$, then $\lim_{k\to\infty}B_k=\infty$. By Kronecker's Lemma (see, for example, \cite[p. 129]{knopp1951theory}), 
			$$\lim_{k\to\infty}\frac{1}{B_k}\sum_{j=0}^kB_j\beta_jR_j=0,$$
			whenever $(\beta_kR_k)\in\ell^1$. It follows that $\lim_{k\to\infty}E_k=0$, and also  $\lim_{k\to\infty}\Delta_k=0$ because $0\le\Delta_k\le E_k$. Part ii) is straightforward from \eqref{E:rate_inexact}.    
		\end{proof}

		\section{A relevant special case}\label{s:special_cases}
		
		In this section, we investigate several special instances of the proposed  algorithm, and discuss practical strategies for solving the corresponding inexact subproblems.
		
		\subsection{The partly smooth case $g\equiv 0$}
		
		Let us discuss the case $g\equiv 0$, which includes the linearly constrained problem \eqref{COP}
		\[\min\{f(x)+F(x):Ax=b\},\]
		when $G(y)=\langle b,y\rangle$.
		
		Assuming (for simplicity but without much loss of generality, since the second step is explicit) that $\epsilon_{k+1}\equiv 0$, 
		\eqref{eq:alg1_composite_inexact} becomes
		\begin{equation}\label{eq:alg_linearcase}
			\left\{\begin{array}{rcl}
				\ds \alpha_k(x_{k+1}-x_k) +\beta_kA^* (y_{k+1}-y_k)+\nabla F(x_k)+\partial f(x_{k+1})+A^*y_{k+1} & \ni & \varepsilon_{k+1} \medskip \\
				\ds -\beta_kA(x_{k+1}-x_k) +\delta_k (y_{k+1}-y_k) -Ax_{k+1} + \nabla G(y_k) & = & 0.
			\end{array}\right. 
		\end{equation}
		From the second substep, we can write
		$$ y_{k+1}=y_k +\frac{\tilde\beta_k}{\delta_k}A(x_{k+1}-x_k)+\frac{1}{\delta_k}Ax_k-\frac{1}{\delta_k}\nabla G(y_k).$$
		Substituting this in the first one, we get
		$$\left(\alpha_kI + \frac{\tilde\beta_k^2}{\delta_k}A^*A\right)x_{k+1}+\partial f(x_{k+1})\ni \tilde x_k+\varepsilon_{k+1},$$
		where we have written
		$$\tilde x_k:= \alpha_kx_k-\frac{\tilde\beta_k(\tilde\beta_k+1)}{\delta_k}A^*Ax_k-\nabla F(x_k)+\frac{\tilde\beta_k}{\delta_k}A^*\nabla G(y_{k})-A^*y_k.$$
		In other words,
		$$x_{k+1}=\left(\alpha_kI + \frac{\tilde\beta_k^2}{\delta_k}A^*A+\partial f\right)^{-1}\big(\tilde x_k+\varepsilon_{k+1}\big),$$
		which also gives
		\begin{equation} \label{eq:prox_inexact}
			\left\|x_{k+1}-\left(\alpha_kI + \frac{\tilde\beta_k^2}{\delta_k}A^*A+\partial f\right)^{-1}\tilde x_k\right\|\le e_k,
		\end{equation}
		with $e_k\ge \|\varepsilon_{k+1}\|$. 
		
		This falls in the framework of variable metric proximal operations \cite{bonnans1995family}. 
		
		\begin{remark}
			A similar construction is possible when $f\equiv 0$.
		\end{remark}

		\subsection{Implementation issues} \label{SS:Semismooth}
		
		In our numerical experiments, the approximate computation \eqref{eq:prox_inexact} will be relevant. We briefly discuss one way to ensure it, for completeness. 
		
		Let us consider the iterations defined  as in Remark \ref{r:iterations_shift}, 
		\begin{equation}\label{eq:alg-semi-im}
			\left\{\begin{array}{rcl}
				\ds \alpha_k(x_{k+1}-x_k) +\tilde\beta_kA^* (y_{k+1}-y_k)+\nabla F(x_k)+\partial f(x_{k+1})+A^*y_k & \ni & \varepsilon_{k+1} \medskip \\
				\ds -\tilde\beta_kA(x_{k+1}-x_k) +\delta_k (y_{k+1}-y_k) -Ax_{k} + \nabla G(y_{k}) & = & 0.
			\end{array}\right. 
		\end{equation}
		The primal subproblem appearing in the update of $x_{k+1}$ given by the first equation in \eqref{eq:alg-semi-im} does not admit a closed-form solution in general: notice that it depends both on $x_{k+1}$ and $y_{k+1}$. Then using the expression for $x_{k+1}$ given by the first equation, we can replace it on the second one, which gives a nonlinear equation in $y_{k+1}$. We therefore employ a semi-smooth Newton method (see \cite{facchinei2003finite}) to solve for $y_{k+1}$, and then compute $x_{k+1}$ via the proximity operator. The analysis closely parallels that of \cite{luo2022primal}, but we summarize the more general derivation here for completeness. Define $\eta_{k}=\frac{1}{\alpha_{k}}$, and use the first step of algorithm \eqref{eq:alg-semi-im}, to get
		\[x_{k+1}=\operatorname{prox}_{\eta_{k}f}\left(x_{k}-\eta_{k}\tilde\beta_kA^*y_{k+1}+\eta_{k}(\tilde\beta_k-1)A^*y_{k}-\eta_{k}\nabla F(x_{k})+\eta_k\varepsilon_{k+1}\right).\]
		Denote $v_{k}=x_{k}+\eta_{k}(\tilde\beta_k-1)A^*y_{k}-\eta_{k}\nabla F(x_{k})+\eta_k\varepsilon_{k+1}$, and  $w_{k}=\delta_{k}y_{k}-(\tilde\beta_{k}-1)Ax_{k}-\nabla G(y_{k})$. Replacing the proximal formulation for $x_{k+1}$ in the second step of \eqref{eq:alg-semi-im}, we have
		\begin{equation}\label{eq:subproblem}
			\delta_{k}y_{k+1}-\tilde\beta_{k}A\operatorname{prox}_{\eta_{k}f}\left(v_{k}-\eta_{k}\tilde\beta_k A^*y_{k+1}\right)-w_{k}=0.    
		\end{equation}
		
		We can define a mapping:
		\begin{equation}\label{eq:mappingf}
			H_{k}(y)=\delta_{k}y-\tilde\beta_{k}A\operatorname{prox}_{\eta_{k}f}\left(v_{k}-\eta_{k}\tilde\beta_kA^*y\right)-w_{k},
		\end{equation}
		
		so that solving \eqref{eq:subproblem} is equivalent to finding a zero of $H_k$. Let $\partial \operatorname{prox}_{\eta_{k}f}(y)$ be the generalized Clarke subdifferential of the Lipchitz mapping $\operatorname{prox}_{\eta_{k}f}(y)$. 
		If $P_{k}(y)\in\partial\operatorname{prox}_{\eta_{k}f}\left(v_{k}-\eta_{k}\tilde\beta_k A^*y\right)$ is symmetric, then for any $y\in \mathbb{R}^{m}$, by chain rule, we have the generalized Jacobian of $H_k$ given by
		\[JH_{k}(y):=\delta_{k}I+\eta_{k}\tilde\beta_k^2AP_{k}(y)A^*\in\mathbb{R}^{m\times m}.\]
		Using the previous, we can implement a semi-smooth Newton iteration for solving $H_{k}(y)=0$: given an initial guess $y_{0}\in\mathbb{R}^{m}$, perform the iteration
		\[y^{j+1}=y^{j}-[JH_{k}(y^{j})]^{-1}H_{k}(y^{j}), \quad j \in \N.\]
		Notice that, using Moreau's identity, $\operatorname{prox}_{\eta f}(x)+\eta\operatorname{prox}_{f^*/\eta}(x/\eta)=x,$ where $f^*$ is the conjugate function of $f$, we find that $H_{k}(y)=\nabla \mathcal{H}_{k}(y)$, with $\mathcal{H}_{k}(\cdot)$ being defined by
		\begin{equation}\label{eq:capF}
			\begin{aligned}
				\mathcal{H}_{k}(y):=&\frac{\delta_{k}}{2}\|y\|^2-\langle w_{k},y\rangle+f^*\left(\operatorname{prox}_{f^*/\eta_{k}}(v_{k}/\eta_{k}-\tilde\beta_kA^*y)\right)\\
				&+\frac{1}{2\eta_{k}}\left\|\operatorname{prox}_{\eta_{k}f}\left(v_{k}-\eta_{k}\tilde\beta_kA^*y\right)\right\|^2.    
			\end{aligned}    
		\end{equation}
		The proof of this derivation can be found in Appendix \ref{appendix:function_H}. Notice that finding a zero of $H_k$ is equivalent to minimize a convex function $\mathcal{H}_k$. Then, to guarantee global convergence, the semi smooth Newton iteration is combined with a backtracking line-search procedure \cite{dennis1996numerical}. Given the Newton direction
		\[d^{j}=-[JH_{k}(y^{j})]^{-1}H_{k}(y^{j}),\quad j \in \N.\]
		find the smallest nonnegative integer $r\in \N$ such that
		\[\mathcal{H}_k(y^{j}+\rho^{r}d^{j})\leq\mathcal{H}_k(y^{j})+\nu\rho^{r}\langle H_{k}(y^{j}),d^{j}\rangle,\]
		where $\nu\in(0,1/2)$, $\rho\in(0,1]$. Summarizing: Given $x_0\in\mathcal X$, $y_0\in\mathcal Y$, and sequences
		$(\alpha_k)$, $(\beta_k)$, and $(\delta_k)$ which satisfy Assumption \ref{a:a_composite_inexact}, we iterate for $k \in \N$ as
		\begin{align*}
			\eta_k &= \frac{1}{\alpha_k},\\
			v_k &=x_{k}+\eta_{k}(\tilde\beta_k-1)A^*y_{k}-\eta_{k}\nabla F(x_{k})+\eta_k\varepsilon_{k+1},\\
			w_k &=\delta_{k}y_{k}-(\tilde\beta_{k}-1)Ax_{k}-\nabla G(y_{k}),
		\end{align*}
		and consider the mapping
		\[H_k(y)=\delta_k y-\tilde\beta_k A\prox_{\eta_k f}\!\left(v_{k}-\eta_{k}\tilde\beta_kA^*y\right)-w_k.\]
		Then compute $y_{k+1}$ as an inexact solution of $H_k(y)=0$, and update
		\[x_{k+1}=\operatorname{prox}_{\eta_{k}f}\left(x_{k}-\eta_{k}\tilde\beta_kA^*y_{k+1}+\eta_{k}(\tilde\beta_k-1)A^*y_{k}-\eta_{k}\nabla F(x_{k})+\eta_k\varepsilon_{k+1}\right).\]

		\section{Numerical Experiments}\label{s:numerics}
		In this section, we test numerically iterations  \eqref{eq:alg1_composite_inexact} in different optimization problem settings. We focus on the particular instances described in Section \ref{s:special_cases}. 
		\subsection{Linearly constrained case}
		
		Consider the \textit{linearly constrained least squares} problem
		\begin{equation}\label{eq:prob_linear_numerics}
			\min \frac{1}{2} \Vert Bx - c \Vert^2, \quad \text{s.t.} \quad Ax =b,
		\end{equation}
		with $B \in \mathbb{M}_{r\times n}$, $A\in \mathbb{M}_{m\times n}$, $b \in \R^m$, $c \in \R^n$. We will test the inexact minimization subroutine described in Section \ref{SS:Semismooth}, in the smooth objective case, that is $f\equiv 0$ and $F(\cdot) = \frac{1}{2} \Vert B(\cdot) - c \Vert^2$. As discussed previously, at each iteration $k$, the update $x_{k+1}$ is obtained by approximately minimizing an auxiliary function $\Psi_k(x)$ up to a prescribed accuracy $\varepsilon_k$. If this error is properly controlled, the convergence of the sequence $(x_k,y_k)$ is guaranteed. In our implementation, the error is measured by $\varepsilon_k = \Vert\nabla \Psi_k(v_k)\Vert$, where $v_k$ denotes the iterate produced by the inner subroutine. The stopping criterion of this subroutine is chosen so that this error satisfies the assumptions of Theorem~\ref{T:implicit_inexact}.

		Let us focus on the parameter selection described in Example \ref{ex:parameters_errors}, that is, $\alpha_k \equiv \alpha_0$, $\delta_k = \delta_0$, and $\beta_k = \beta_0 + k$. Theorem \ref{T:implicit_inexact} guarantees convergence if  $\sum_{k \in \N}k\varepsilon_k < +\infty$. Let: 
		\begin{equation}\label{eq:psi_linear_numerics}
			\Psi_k(v) = \frac{1}{2} \Vert Bv - c \Vert^2 + \inner{y_k}{Av - b} + \frac{\alpha_k}{2}\norm{v - z_k}^2 + \frac{1}{2\delta_k}\norm{(\beta_k+1)Av - b}^2,   
		\end{equation}
		with $z_k = \left(I_n + \frac{\beta_k}{\delta_k\alpha_k}(\beta_k+1)A^\top A\right)x_k$. Since $\Psi_k$ is strongly convex and quadratic, computing its minimizer is equivalent to solving a symmetric positive definite linear system. Instead of solving this system exactly, we use a conjugate gradient subroutine to compute an approximate solution satisfying a prescribed tolerance. The inner conjugate gradient routine is presented in Algorithm~\ref{alg:alg_conj}. Then we can state our algorithm as Algorithm \ref{alg:alg_linear_inexact}. In what follows, we will consider a problem where $r$ is considerably smaller than $n$, and we will construct an ill-conditioned objective and constraint operators. 
		Such instances are commonly used in numerical studies since ill-conditioning tends to amplify differences in stability and robustness between optimization algorithms. Matrices $A \in \R^{m\times n}$ and $B \in \R^{r\times n}$ are built from random orthogonal matrices and prescribed singular values. More precisely, random Gaussian matrices are orthogonalized through QR factorizations, while the singular values are chosen as logarithmically spaced sequences between 1 and $10^{-3}$. The vectors $b$ and $c$ defining the objective and the feasible set are generated independently in order to obtain a nontrivial solution. For this purpose, we generate two independent Gaussian vectors $x_{\text{obj}}$ and $x_{\text{feas}}$ and then 
		\[c = Bx_{\text{obj}}, \quad b=Ax_{\text{feas}}.\]
		Then, the minimizer of the unconstrained problem is in general, not feasible for the linear constraints. 
		
		\begin{algorithm}[h]
			\caption{Conjugate gradient routine for minimizing $\Psi_k$}
			\label{alg:cg_inner}
			\begin{algorithmic}[1]
				
				\STATE \textbf{Input:}
				$x_k$, $y_k$, $\beta_k$, tolerance $\varepsilon_k$
				
				\STATE Define
				\[
				H_k
				=
				B^\top B
				+
				\alpha_k I_n
				+
				\frac{(\beta_k+1)^2}{\delta_k}A^\top A
				\]
				
				\STATE Define
				\[
				\nabla\Psi_k(v)=H_kv-h_k
				\]
				
				\STATE Initialize $v_0=x_k,
				\,
				r_0=-\nabla\Psi_k(v_0),
				\,
				p_0=r_0, \, j=0$
				\WHILE{$\|r_j\|>\varepsilon_k$}
				
				\STATE $\sigma_j
				=
				\frac{\langle r_j,r_j\rangle}
				{\langle p_j,H_kp_j\rangle}$

				\STATE $v_{j+1}=v_j+\sigma_jp_j$

				\STATE $r_{j+1}=r_j-\sigma_jH_kp_j$

				\STATE $\tau_j
				=
				\frac{\langle r_{j+1},r_{j+1}\rangle}
				{\langle r_j,r_j\rangle}$

				\STATE $p_{j+1}=r_{j+1}+\tau_jp_j$

				\STATE $j\leftarrow j+1$
				
				\ENDWHILE
				
				\RETURN $v_{j+1}$
				
			\end{algorithmic}
			\label{alg:alg_conj}
		\end{algorithm}

		\begin{algorithm}[h]
			\caption{Inexact least squares - Implementation of \eqref{eq:alg_linearcase} using conjugate gradients.}
			\label{alg:outer_linear_inexact}
			\begin{algorithmic}[1]
				\STATE \textbf{Given:}
				$x_0\in\mathbb{R}^n$,
				$y_0\in\mathbb{R}^m$,
				$\alpha_0,\delta_0,\beta_0>0$,
				$N\in\mathbb{N}$,
				$\varepsilon_0>0$
				
				\FOR{$k=0,\dots,N-1$}
				
				\STATE $\varepsilon_k = \dfrac{\varepsilon_0}{(k+1)^{2.1}}$
				
				\STATE Construct $\Psi_k$ as in \eqref{eq:psi_linear_numerics}. 
				
				\STATE Compute
				\[
				x_{k+1} \approx \arg\min_v \Psi_k(v)
				\]
				using Algorithm~\ref{alg:cg_inner}
				with tolerance $\varepsilon_k$
				
				\STATE
				\[
				y_{k+1}
				=
				y_k
				+
				\frac1{\delta_k}
				\Big(
				\beta_k A(x_{k+1}-x_k)
				+
				Ax_{k+1}
				-b
				\Big)
				\]
				
				\STATE $\beta_{k+1} = \beta_k + 1$
				
				\ENDFOR
				
				\RETURN $(x_N,y_N)$
			\end{algorithmic}
			\label{alg:alg_linear_inexact}
		\end{algorithm}

		We perform the experiment considering $n=500$, $m=50$, $r=100$. First, we look to find the best values of $\alpha_0$ and $\delta_0$. The Heatmap in Figure \ref{fig:heatmap1} shows the value of $\Delta_k$ for different values of the parameters, evaluated at the last of 100 iterations. This suggests that $\alpha_0$ should be small, so we refine the grid and obtain the Heatmap displayed in Figure \ref{fig:heatmap2}. Hence, we obtain that the smallest value for the Dual gap is achieved for $\alpha_0=10^{-1}$, $\delta_0=3$. 
		
		\begin{figure}[h]
			\centering
			\subfigure[$\alpha_0 \in \lbrack 1 , 5.25\rbrack$, $\delta_0 \in \lbrack1,15\rbrack$]{\includegraphics[width=0.48\textwidth]{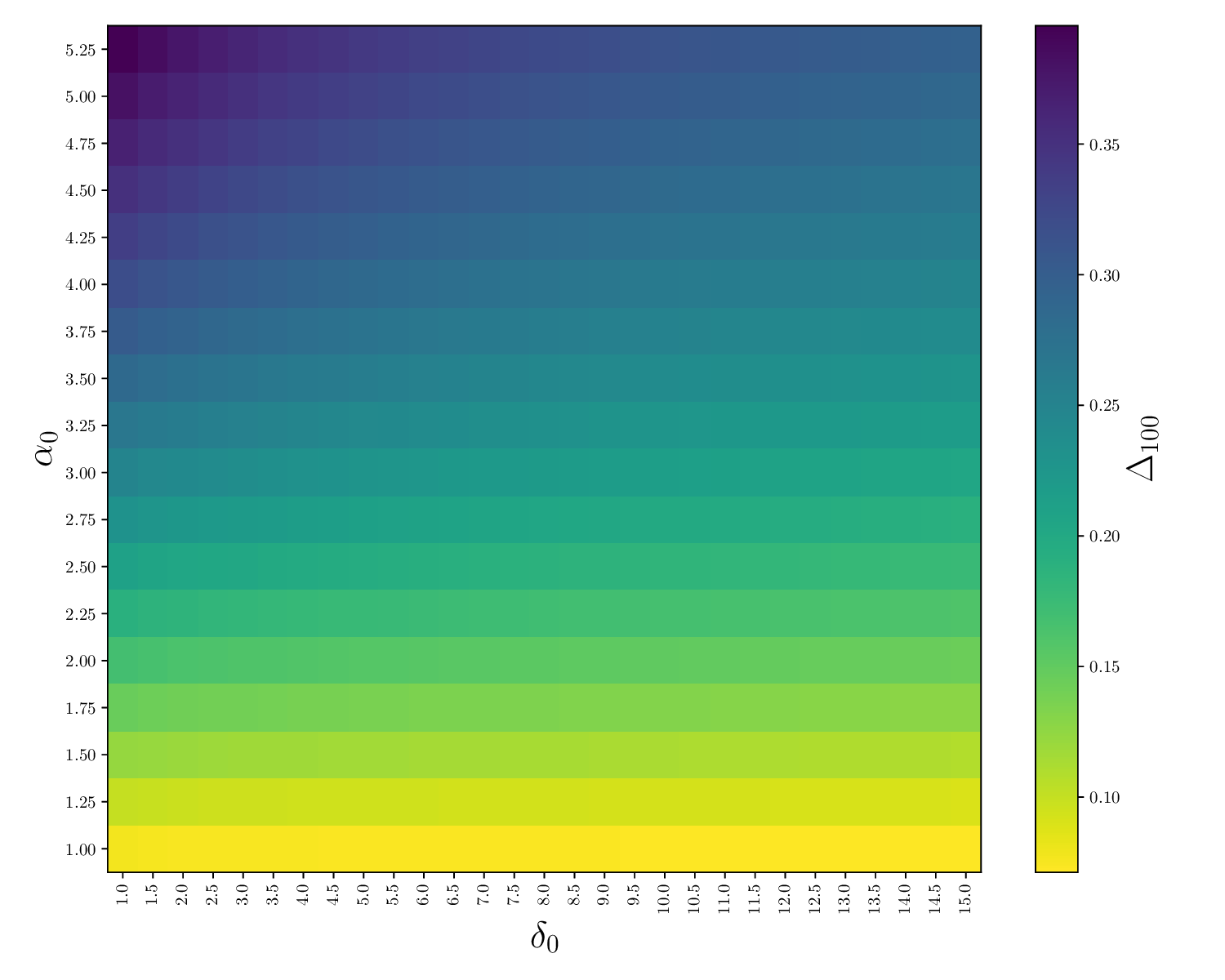}\label{fig:heatmap1}}
			\subfigure[$\alpha_0 \in \lbrack0.1, 0.15\rbrack$, $\delta_0 \in \lbrack1,5\rbrack$]{\includegraphics[width=0.48\textwidth]{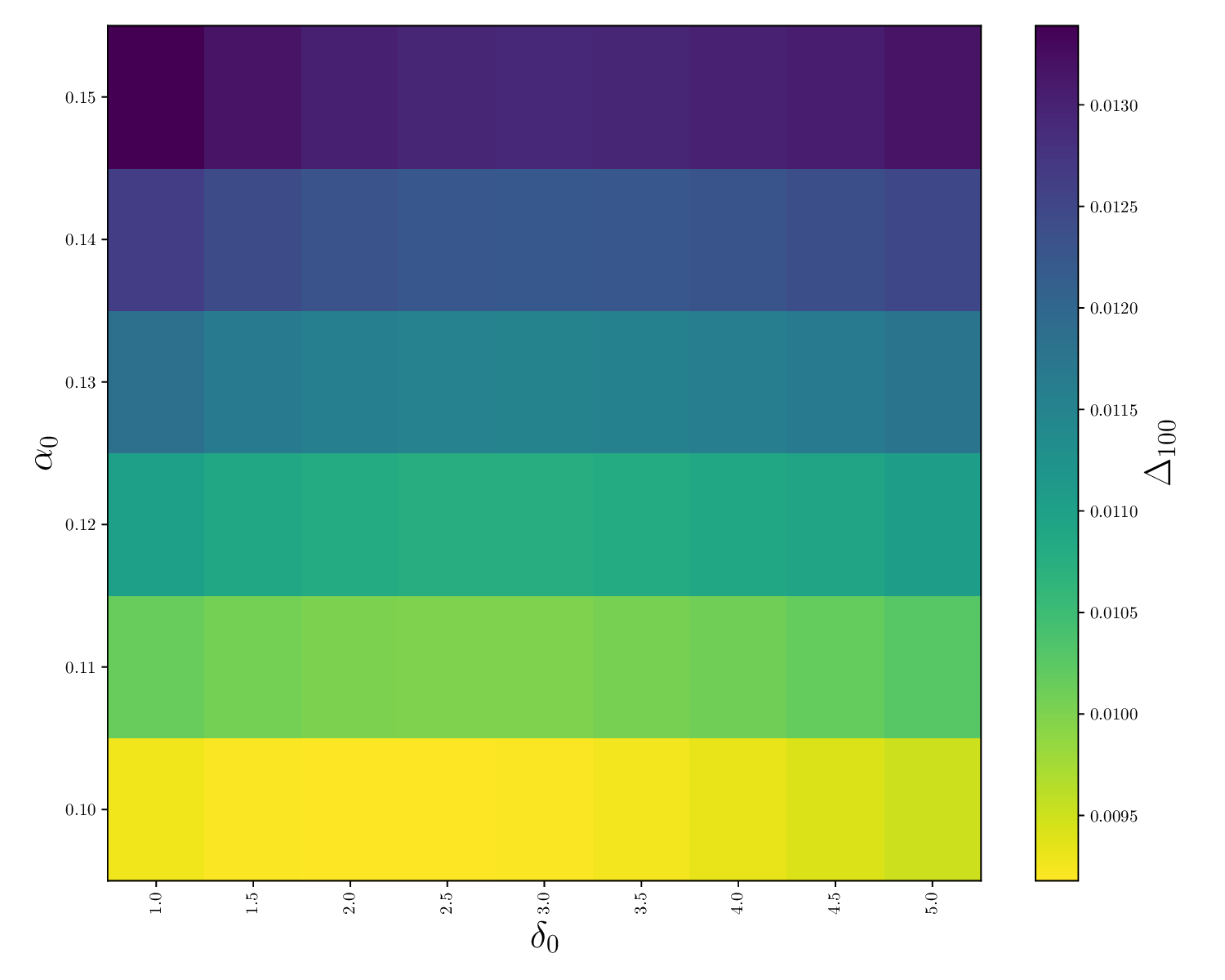}\label{fig:heatmap2}}
			\caption{The Heatmaps display the value of $\Delta_{100}$, for different choices of $\alpha_0$, $\delta_0$ for Algorithm \ref{alg:alg_linear_inexact} applied to problem \eqref{eq:prob_linear_numerics}.}
			\label{fig:heatmap}
		\end{figure}

		We test the algorithm using $\beta_0 = 1$, $\alpha_0=\delta_0=1$, and also the optimized parameters found before, $\alpha_0 =10^{-1}$, $\delta_0=3$. As a benchmark, we compare our algorithm with Chambolle-Pock \cite{chambolle2011first}, and we display the results obtained in Figure \ref{fig:inexact_leastsq}, for different optimality measures. For the Chambolle Pock algorithm the resolvent on the primal iteration is computed explicitly.  
		
		\begin{figure}[h]
			\centering
			\subfigure[Objective values]{\includegraphics[width=0.48\textwidth]{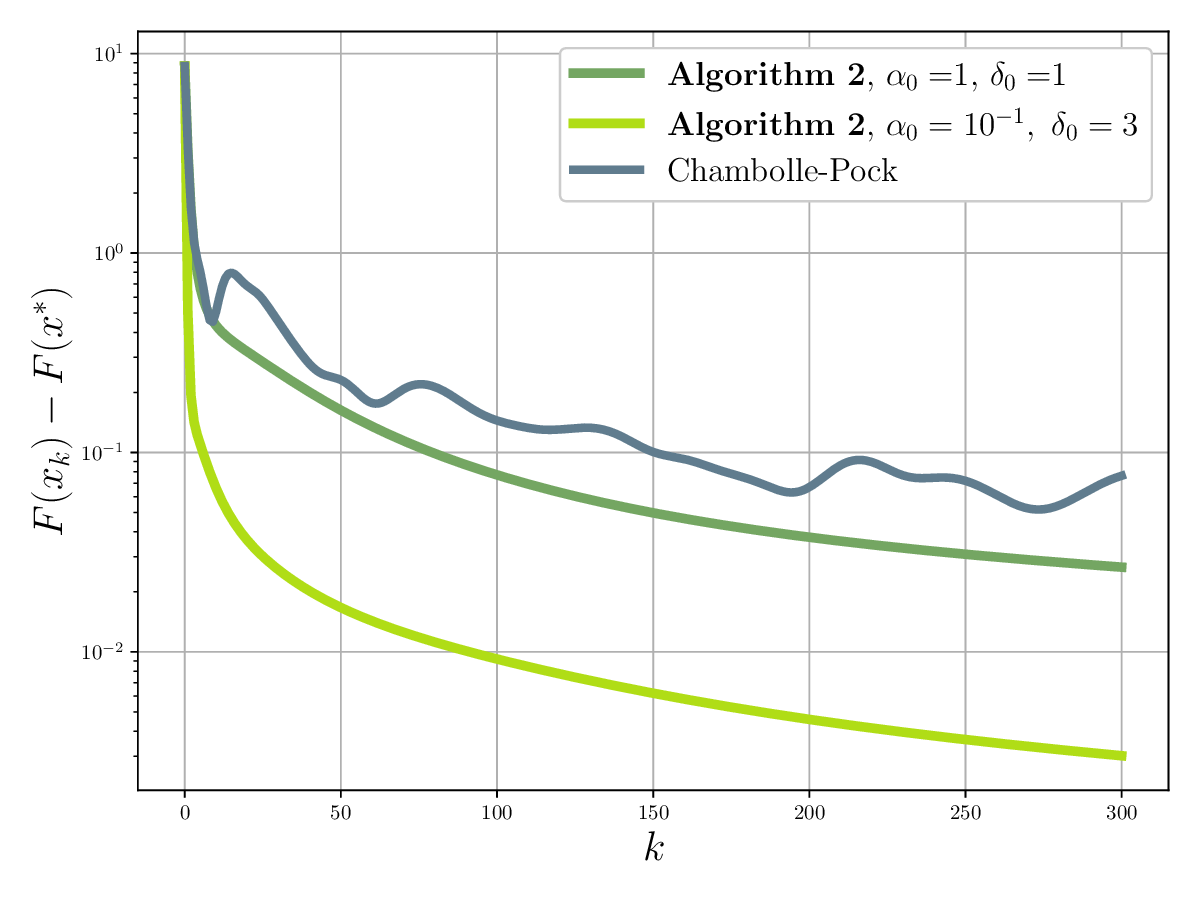}}
			\subfigure[Optimality residual]{\includegraphics[width=0.48\textwidth]{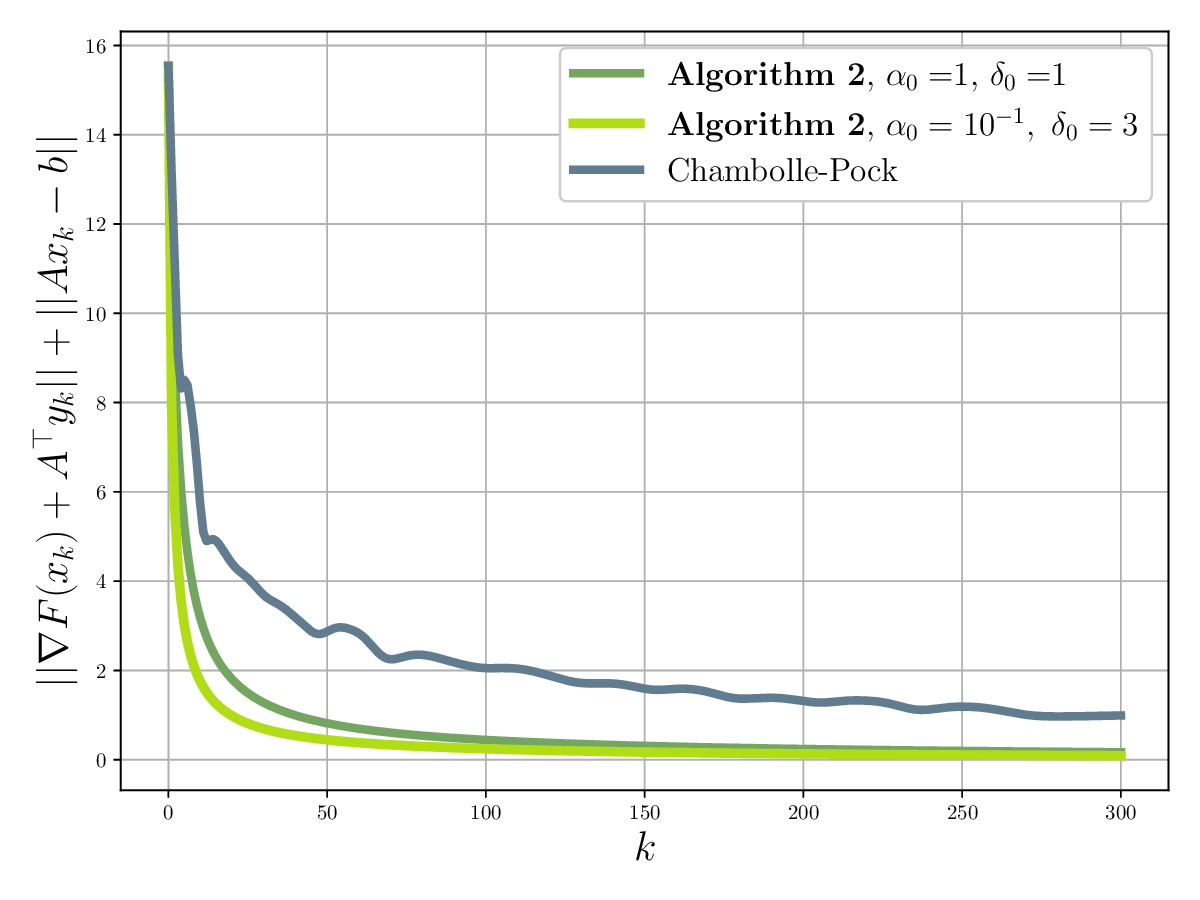}}\\
			\subfigure[Feasibility measure]{\includegraphics[width=0.48\textwidth]{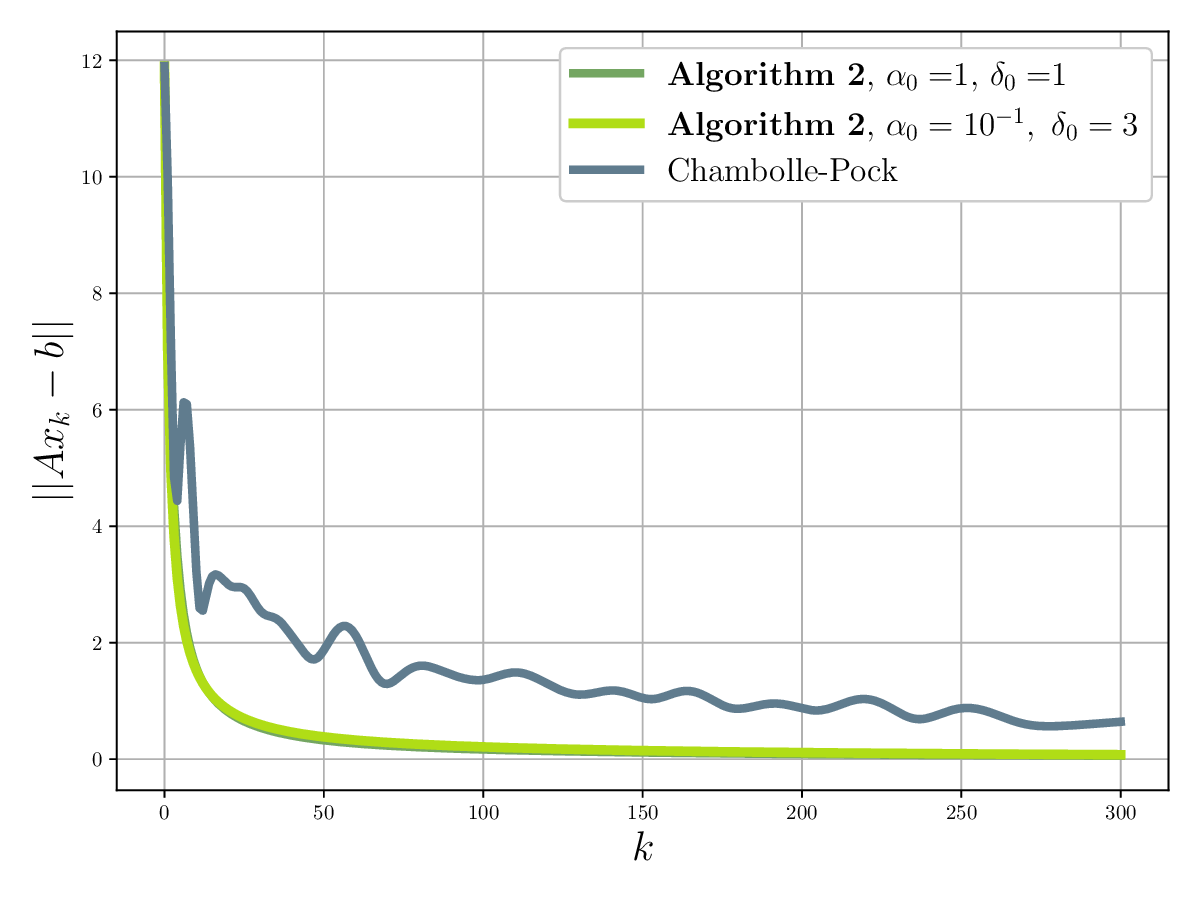}}
			\subfigure[Dual gap]{\includegraphics[width=0.48\textwidth]{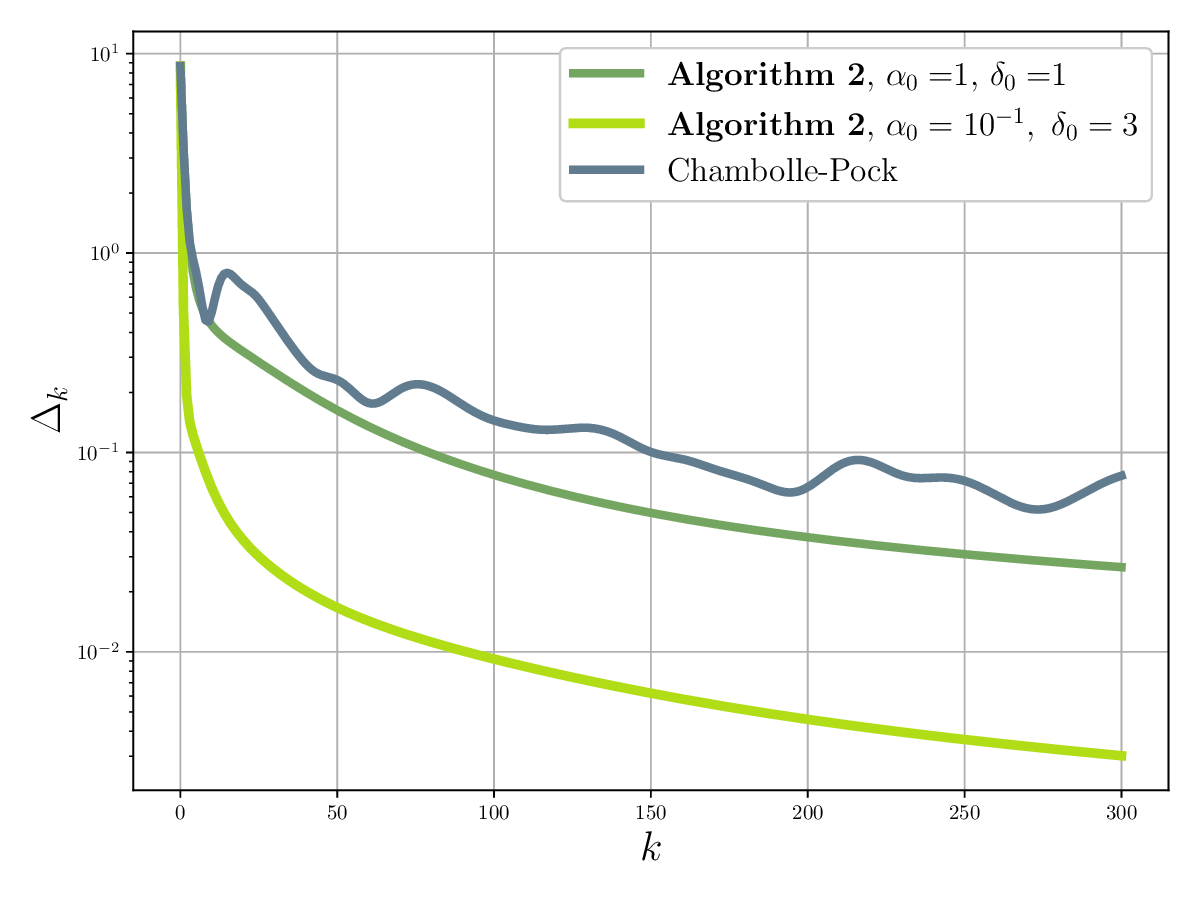}}
			\caption{Linearly constrained least squares}
			\label{fig:inexact_leastsq}
		\end{figure}
		\subsection{Linear constrained composite case}
		We consider the linearly constrained $\ell_{1}-\ell_{2}$ problem:
		\begin{equation}\label{eq:linear_semi_composite}
			\min_{x\in\mathbb{R}^{n}}\frac{\lambda}{2}\|x\|^2+\|x\|_{1},\quad\text{s.t.}\quad Ax=b,    
		\end{equation}
		where $\lambda>0$,   $A\in\mathbb{R}^{m\times n}$, and $b \in \R^m$ with $m\ll n$. This setting can be interpreted as a compressed sensing instance with a smooth regularization given by $\lambda$ \cite{candes2008introduction}. In the following, we solve this problem via a semi smooth Newton primal dual algorithm following the discussion in Section \ref{SS:Semismooth}.

		In this case, since $f(x)= \Vert x\Vert_1$ is piecewise affine, $\operatorname{prox}_{\eta f}$ is strongly semi smooth and so is the nonlinear mapping $H_{k}$ defined in \eqref{eq:mappingf}. For $\eta>0$ and $x\in\mathbb{R}^{n}$, define a diagonal matrix
		\begin{equation}\label{eq:diagmatrix}
			P_{\eta}(x)=\operatorname{diag}(p)\in\mathbb{R}^{n\times n}\quad \text{with}\quad p_{i}=\left\{\begin{aligned}
				&1\quad\text{if}\quad |x_{i}|\geq \eta,\\ 
				&0\quad\text{if}\quad |x_{i}|<\eta.
			\end{aligned} \right.
		\end{equation}
		Then it is easy to see that $P_{\eta}(x)\in \partial\operatorname{prox}_{\eta f}(x)$, and thus, we have a generalized Clarke subgradient for \eqref{eq:subproblem}
		\[JH_{k}(y)=\delta_{k}I+\eta_{k}\beta_{k}^2AP_{\eta_{k}}[q_{k}(y)]A^*\in\mathbb{R}^{m\times m},\]
		where $q_{k}(y)=v_{k}-\eta_{k}\beta_{k}A^*y$.  Moreover, in this context, the function $\mathcal{H}_k$ defined in \eqref{eq:capF} becomes
		\[\mathcal{H}_{k}(y)=\frac{\delta_{k}}{2}\|y\|^{2}-\langle w_{k},y\rangle+\frac{1}{2\eta_{k}}\|\operatorname{prox}_{\eta_{k}f}[q_{k}(y)]\|^2.\]
		We consider the stopping criterion:
		\[\text{Res}(k):\max\{\text{Res}(x_{k}),\text{Res}(y_{k})\}\leq \texttt{KKT.Tol},\]
		where the relative residuals are defined by
		\[ \text{Res}(y_{k}):=\frac{\|Ax_{k}-b\|}{1+\|b\|}\quad\text{and}\quad\text{Res}(x_{k}):=\frac{\|x_{k}-\operatorname{prox}_{f}((1-\lambda)x_{k}-A^*y_{k} - \beta_k A^*(y_{k+1}-y_k))\|}{1+\|x_{k}\|}.\]
		The previous construction is not arbitrary: optimality conditions for the problem \eqref{eq:linear_semi_composite} are $Ax^*=b$ and $x^*=\operatorname{prox}_{f}((1-\lambda)x^*-A^Ty^*)$. Then, the residual $\text{Res}(x_{k})$ resembles the proximal step carried out to compute $x_{k+1}$, and can be interpreted as the optimality condition with a correction term given by the dual update $(y_{k+1} - y_k)$. 
		
		The inner iteration is stopped either some fixed tolerance is achieved, meaning $\|H_{k}(y)\|\leq$ \texttt{SsN.Tol} or a maximum of $j_{\max}$ iterations is performed. Considering all of the previous, we design Algorithm \ref{a:composite_ssn}. 
		
		\begin{algorithm}
			\caption{Linear constrained semi smooth Newton}
			\label{a:composite_ssn}
			\setstretch{1.3}
			\begin{algorithmic}[1]
				\STATE \textbf{Given:}
				$x_0\in\mathbb{R}^n$,
				$y_0\in\mathbb{R}^m$,
				$\alpha_0,\delta_0,\beta_0>0$. Define the tolerances:  \texttt{SsN.Tol} and \texttt{KKT.Tol}. 
				\FOR{$k=0,1,\dots$} 
				\STATE Update $\alpha_k$, $\beta_k$, $\delta_k$ 
				.
				\STATE Set $\eta_{k}=\frac{1}{\alpha_{k}}$.
				\STATE Set $v_{k}=x_{k}-\eta_{k}\lambda x_{k}+\eta_k\beta_{k}A^*y_{k}$.
				\STATE Set $w_{k}=\delta_{k}y_{k}-\beta_{k}Ax_{k}-b$.
				\STATE Solve $y_{k+1}$ from the nonlinear equation
				\[H_{k}(y)=\delta_{k}y-(\beta_{k}+1)A\operatorname{prox}_{\eta_{k}f}\left(v_{k}-\eta_{k}(\beta_{k}+1)A^*y\right)-w_{k}\]
				via the following inner loop with $y=y_{k}$ and $j=0$:
				\STATE \textbf{while} $\|H_{k}(y)\|>\texttt{SsN.Tol}$ and $j<j_{max}$ \textbf{do}
				\STATE \quad Compute $q_{k}=v_{k}-\eta_{k}(\beta_{k}+1)A^*y$.
				\STATE \quad Find $P_{\eta_{k}}(q_{k})\in\partial \operatorname{prox}_{\eta_{k}f}(q_{k})$ via \eqref{eq:diagmatrix}.
				\STATE \quad Compute $JH_{k}(y)=\delta_{k}I+\eta_{k}(\beta_{k}+1)^2AP_{\eta_{k}}(q_{k})A^*$.
				\STATE \quad Solve $H_{k}(y)d=-H_{k}(y)$.
				\STATE \quad Find the smallest integer $r$ such that $\mathcal{H}_{k}(y+\rho^{r}d)\leq\mathcal{H}_{k}(y)+\nu\rho^{r}\langle H_{k}(y),d\rangle$.
				\STATE \quad Update $y=y+\rho^{r}d$ and $j=j+1$.
				\STATE \textbf{end while}
				\STATE Update $y_{k+1}=y$ and $x_{k+1}=\operatorname{prox}_{\eta_{k}f}\left(v_{k}-\eta_{k}(\beta_{k}+1)A^*y_{k+1}\right)$.
				\STATE \textbf{if} Res(k)$\leq$ \texttt{KKT.Tol} \textbf{then}
				\STATE \quad \textbf{break}
				\STATE \textbf{end if}
				\ENDFOR
				\RETURN $x_{k+1}$, $y_{k+1}$
			\end{algorithmic}
		\end{algorithm}

		We test the Algorithm and compare it with the iterations given by 
		\begin{equation}\label{eq:sist_luo}
			\left\{\begin{array}{rcl}
				\ds \alpha_k(x_{k+1}-x_k) +\nabla F(x_k)+\partial f(x_{k+1})+A^*y_{k+1} & = & 0, \medskip \\
				\ds -\beta_kA(x_{k+1}-x_k) +\delta_k (y_{k+1}-y_k) -Ax_{k} + b & = & 0.
			\end{array}\right. 
		\end{equation}
		which corresponds to the method studied in \cite{luo2022primal}. Notice that, up to a mismatch in an index of the second equation, this method can be seen as iterations \eqref{eq:alg1_composite_inexact} but neglecting the term on the first equation multiplied by $\beta_k$. As in \cite{luo2022primal}, iterations \eqref{eq:sist_luo} are implemented by following a Semi-Newton routine.

		The algorithms are evaluated on an ill-conditioned, feasible linear constraint system $Ax = b$ of dimensions $m = 100$ and $n = 500$. The matrix $A$ is randomly generated using a partial singular value decomposition to simulate an ill-posed problem with  controlled singular values as a logarithmically spaced sequence between 1 and $10^{-2}$. We build $x_{\text{true}}$ as a sparse signal featuring 20 non-zero entries sampled from a standard normal distribution, and the vector $b$ is set to $b = Ax_{\text{true}}$ to ensure strict feasibility. 
		We consider $\alpha_k$ and $\beta_k$ satisfying the first two conditions in Assumption \ref{a:a_composite_inexact} with equality, giving
		\[\alpha_k = \dfrac{(\beta_{k}+1)L_F}{2\beta_k+1}, \quad \beta_{k+1} = \frac{1}{2}\left( \beta_k + \sqrt{\beta_k^2 + 2(2\beta_k+1)}\right).\]
		It is not difficult to check, that $\beta_k \sim k$. We will keep $\delta_k \equiv \delta_0$. To have a fair comparison, we will use the same choice for the sequences in both algorithms.  Notice, that \cite{luo2022primal} admits to include the parameter of strong convexity of the smooth objective function $F$ and use it to design the sequences, leading to a linear rate of convergence. We do not consider this case since the strongly convex setting is not studied in the current work. For the line search procedure, we use $\nu=0.2$ and $\rho=0.6$. Starting from randomly generated initial guesses $x_{0}$ and $y_{0}$, $\beta_0=1.1$ and $\delta_0=10^{-4}$, we solve problem \eqref{eq:linear_semi_composite} with $\lambda = 10^{-2}$. We compare our algorithm with \cite{luo2022primal} in Figure \ref{fig:composite_obj_res}.

		\begin{figure}[h]
			\centering
			\subfigure{\includegraphics[width=0.48\textwidth]{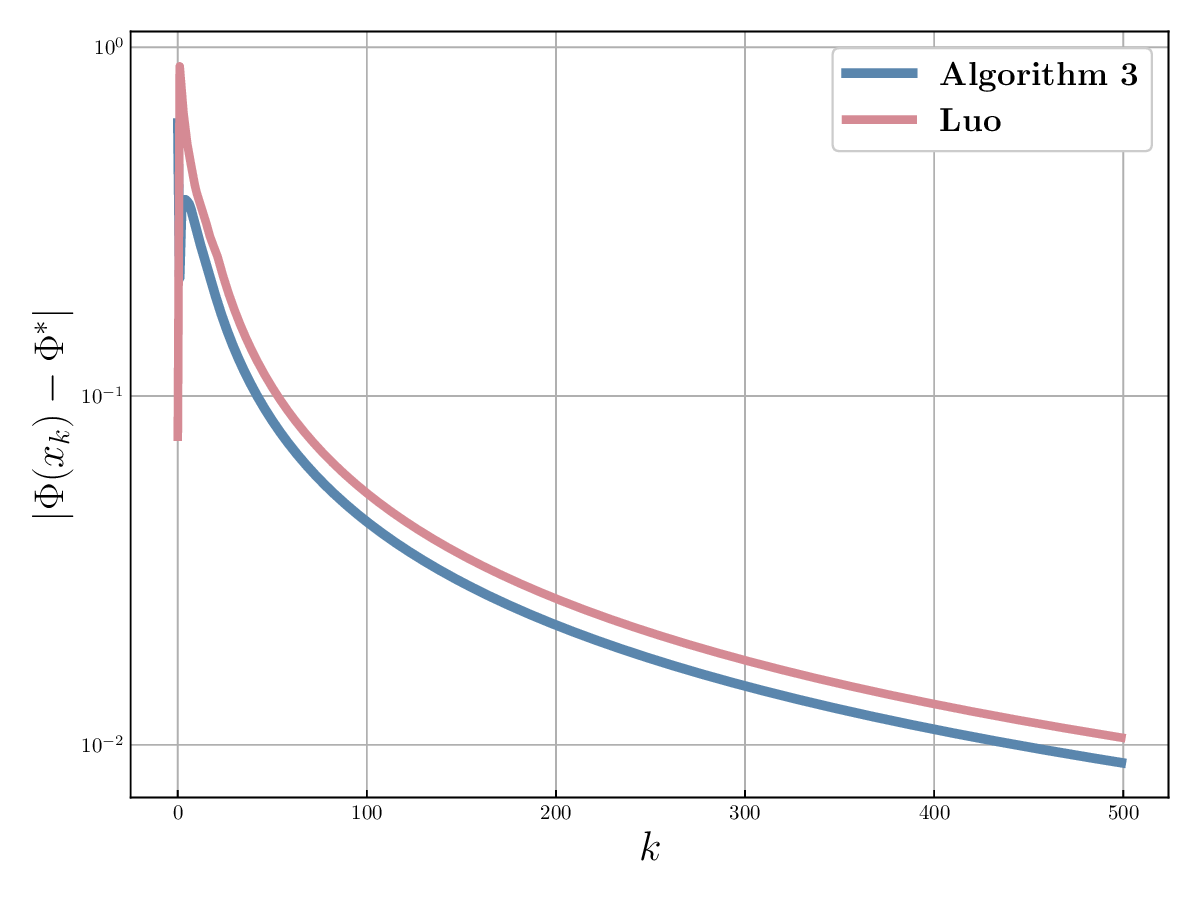}}
			\subfigure{\includegraphics[width=0.48\textwidth]{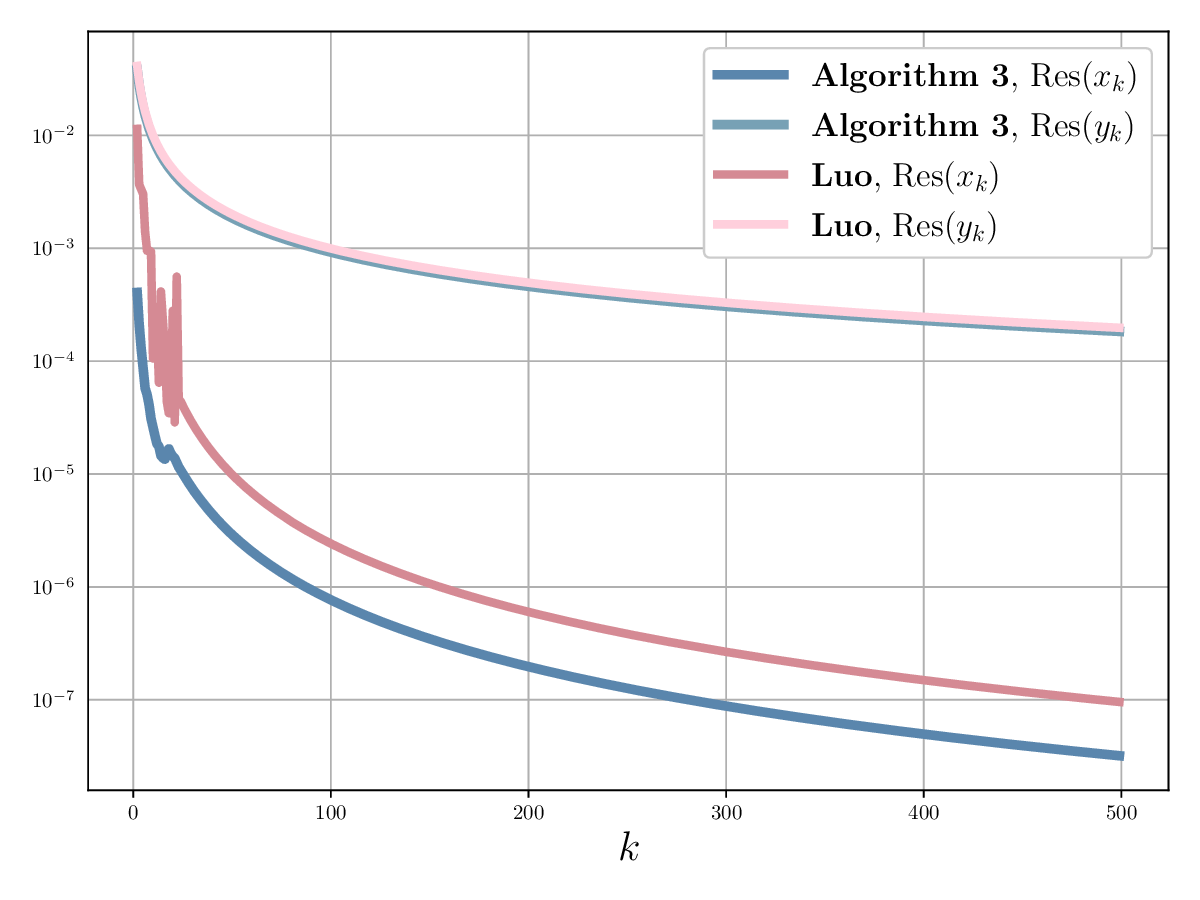}}\\
			\subfigure{\includegraphics[width=0.48\textwidth]{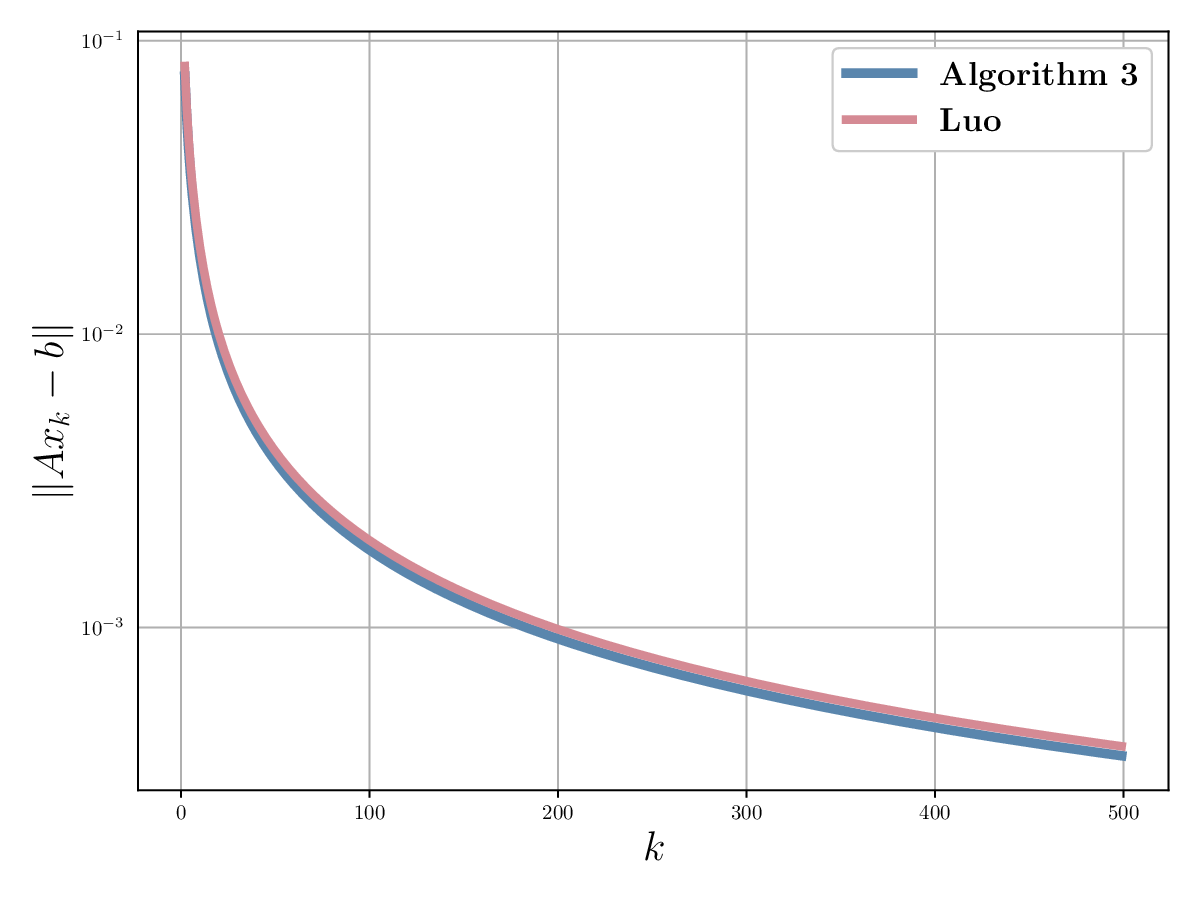}}
			\subfigure{\includegraphics[width=0.48\textwidth]{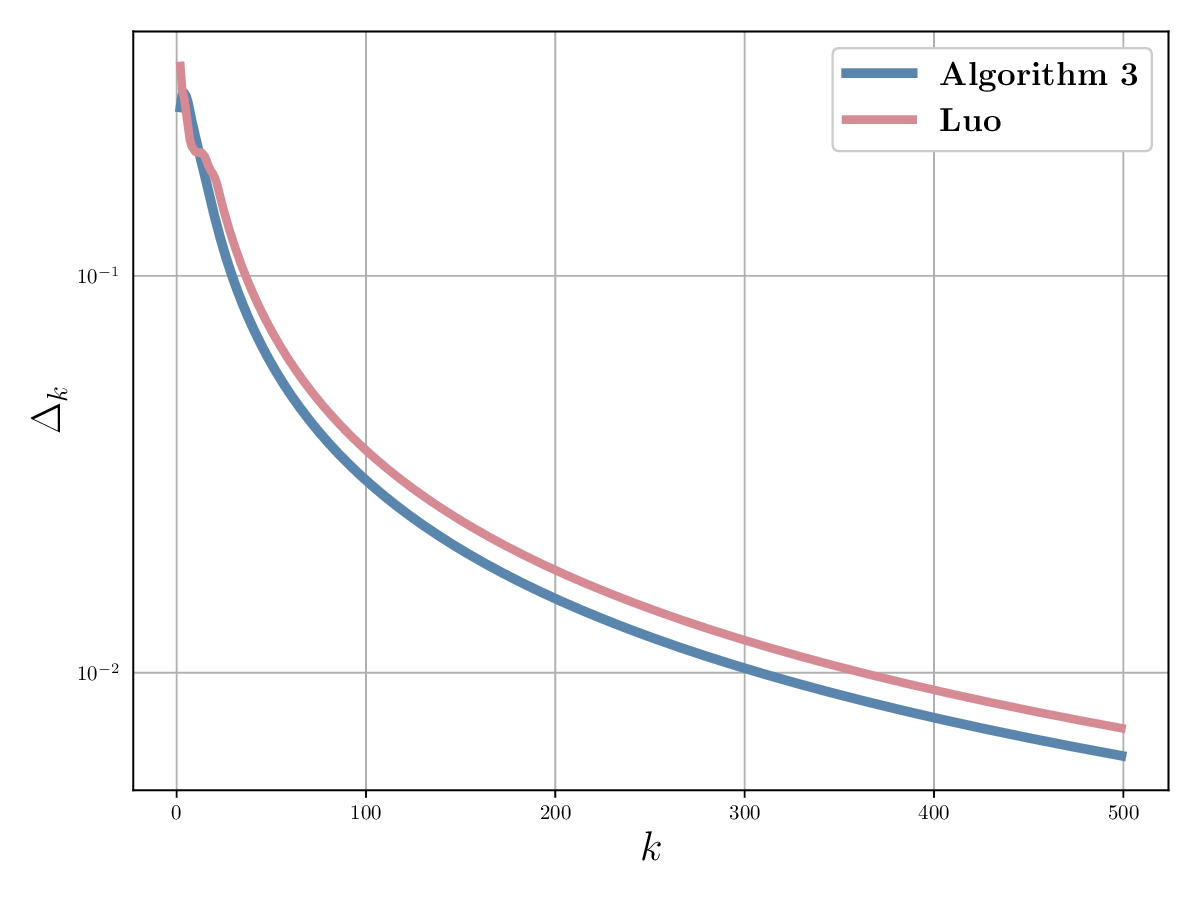}}
			\caption{Performance of Algorithm \ref{a:composite_ssn} for solving \eqref{eq:linear_semi_composite} with $m=100$, $n=500$. We consider $\lambda=10^{-2}$ and we compare the performance of Algorithm \ref{a:composite_ssn} with Luo's method. We plot objective values, Residuals, primal factibility and the Duality gap $\Delta_k$.}
			\label{fig:composite_obj_res}
		\end{figure}
		\subsection{Smooth + Nonsmooth objective}
		
		Consider the problem 
		\begin{equation}\label{eq:prob_wav}
			\min \varphi^\mathcal{W}(x) := \frac{1}{2}\norm{\mathcal{M}x - b}^2_2 + \omega\norm{\mathcal{W}x}_1,   
		\end{equation}
		where $\omega>0$, and $\mathcal{M}$, $\mathcal{W}$ are linear operators. Let $F(\cdot) = \frac{1}{2}\norm{\mathcal{M}\cdot - b}^2_2$, $h(\cdot) = \omega\norm{\cdot}_1$ and $A = \mathcal{W}$. By making $g=h^*$, the previous is a particular instance of \eqref{eq:lag_phi_gamma} with $f \equiv G \equiv 0$. 
		The setting of Problem \ref{eq:prob_wav} is usually used for modeling Image denoising problems. Consider $\bar{x}$ an image to recover from a noisy observation $b$. Let $\mathcal{M}$ be a known blurring operator, and $\mathcal{W}$ as the three-stage Wavelet transform \cite{briceno2023primal}. The scalar parameter $\omega>0$ controls the trade-off between adherence to the observed data $b$ and the regularity (or sparsity) of the reconstruction. Specifically, we consider $\mathcal{M}$ as a Gaussian blur of size $9\times 9$ and standard deviation $4$. The observation is obtained as $b = \mathcal{M}\bar{x} + e$, where $e$ is an additive zero-mean white Gaussian noise with standard deviation $\sigma$. 
		
		Since  $\omega h = \norm{\cdot}_1$, we have that $g=h^* = \iota _C$, where $C = \mathbb{B}_{\norm{\cdot}_{\infty}}(0,\omega)$, that is, the $\omega$-radius ball centered at the origin for the $\ell_{\infty}$ norm. Since $g$ is an indicator function, the proximal operator required in the $y_{k+1}$ update is the projection over $\mathbb{B}_{\norm{\cdot}_{\infty}}(0,\omega)$. By setting $w_ k = \alpha_{k}\delta_{k} + (1+\beta_{k})^2$ and $w_k = \alpha_k \delta_k + \beta_k(\beta_k+1)$, respectively, since $\mathcal{W}\mathcal{W}^\top = I$,
		we can implement Algorithms \eqref{eq:alg1_composite_inexact} and \eqref{eq:alg_1_a}, as Algorithms \ref{alg:alg_xkplus1_wavelets} and \ref{alg:alg_xk_wavelets}, respectively (we are setting $f\equiv 0$ and $G\equiv 0$).
		
		\begin{algorithm}
			\caption{Algorithm \eqref{eq:alg1_composite_inexact} - Wavelets setting}
			\label{alg:alg_xkplus1_wavelets}
			\begin{algorithmic}[1]
				\setstretch{1.2}
				\STATE \textbf{Given:} $x_0\in \R^n$, $y_0 \in \R^m$, and $N \in \mathbb{N}$.
				
				\FOR{$k = 1$ \TO $N-1$}
				\STATE Set $\delta_k\equiv1$, $\beta_k=k+1$ and $\alpha_k=\frac{(\beta_k+1)L_{F}}{2\beta_{k}+1}$.
				\STATE $w_k = \alpha_{k}\delta_{k} + (1+\beta_{k})^2$.
				
				\STATE $y_{k+1} = \text{prox}_{\frac{\alpha_{k}}{w_{k}}g} \left[ \left(1-\frac{(\beta_{k}+1)}{w_{k}}\right)y_{k} + \frac{1}{w_{k}}\mathcal{W}(\alpha_{k}x_{k}-(1+\beta_{k})\nabla F(x_{k})) \right]$.
				
				\STATE $x_{k+1} = x_k - \frac{\beta_k}{\alpha_k}\mathcal{W}^\top(y_{k+1}-y_k) - \frac{1}{\alpha_k}\mathcal{W}^\top y_{k+1} - \frac{1}{\alpha_k}\nabla F(x_k)$.
				\ENDFOR
				
				\RETURN $x_N$, $y_N$
			\end{algorithmic}
		\end{algorithm}
		
		\begin{algorithm}
			\caption{Algorithm \eqref{eq:alg_1_a} - Wavelets setting}
			\label{alg:alg_xk_wavelets}
			\begin{algorithmic}[1]
				\setstretch{1.2}
				\STATE \textbf{Given:} $x_0\in \R^n$, $y_0 \in \R^m$, and $N \in \mathbb{N}$.
				
				\FOR{$k = 1$ \TO $N-1$}
				\STATE Set $\delta_k\equiv1$, $\beta_k=k+1$ and $\alpha_k=\frac{\beta_kL_{F}}{2\beta_{k}-1}$.
				\STATE $w_k = \alpha_k \delta_k + \beta_k(\beta_k+1)$.
				
				\STATE $y_{k+1} = \text{prox}_{\frac{\alpha_{k}}{w_k}g} \parenc{ \paren{1-\frac{\beta_k}{w_k}}y_k + \frac{1}{w_k}\mathcal{W}(\alpha_k x_k - \beta_k \nabla F(x_k)) }$.
				
				\STATE $x_{k+1} = x_k - \frac{\beta_k}{\alpha_k}\mathcal{W}^\top(y_{k+1}-y_k) - \frac{1}{\alpha_k}\mathcal{W}^\top y_{k+1} - \frac{1}{\alpha_k}\nabla F(x_k)$.
				\ENDFOR
				
				\RETURN $x_N$, $y_N$
			\end{algorithmic}
		\end{algorithm}
		
		In what follows, we test both Algorithms \ref{alg:alg_xkplus1_wavelets} and \ref{alg:alg_xk_wavelets} over test images\footnote{Images obtained from \url{http://sipi.usc.edu/database/}} considering $\sigma = 10^{-1}$, giving noisy observations.  The parameters of the algorithm are chosen as $\delta_k \equiv 1$, $\beta_k = 1+k$, and $\alpha_k$ is chosen such that parameters in Assumption \ref{a:a_composite_inexact} hold with equality, where $L = \norm{\mathcal{M}}^2=1$.  Both algorithms produce visually similar recovered images. Therefore, Figure \ref{fig:image} reports the results obtained for three $512\times512$ images after 100 iterations of Algorithm \ref{alg:alg_xk_wavelets}, with $\omega \in \left\lbrace 1, 10 \right\rbrace$. The relatively large value of $\omega=10$ reflects a modeling choice consistent with highly noisy observations: the algorithm is designed to rely less on the potentially corrupted data term and more on the structural prior imposed by the Wavelet regularization.

		\begin{figure}[t]\label{fig:image}
			\centering
			\setlength{\tabcolsep}{2pt}
			
			\begin{tabular}{cccc}
				Original &
				Noisy observation &
				$\omega=1$ &
				$\omega=10$
				\\[2mm]
				
				\includegraphics[width=0.23\textwidth]{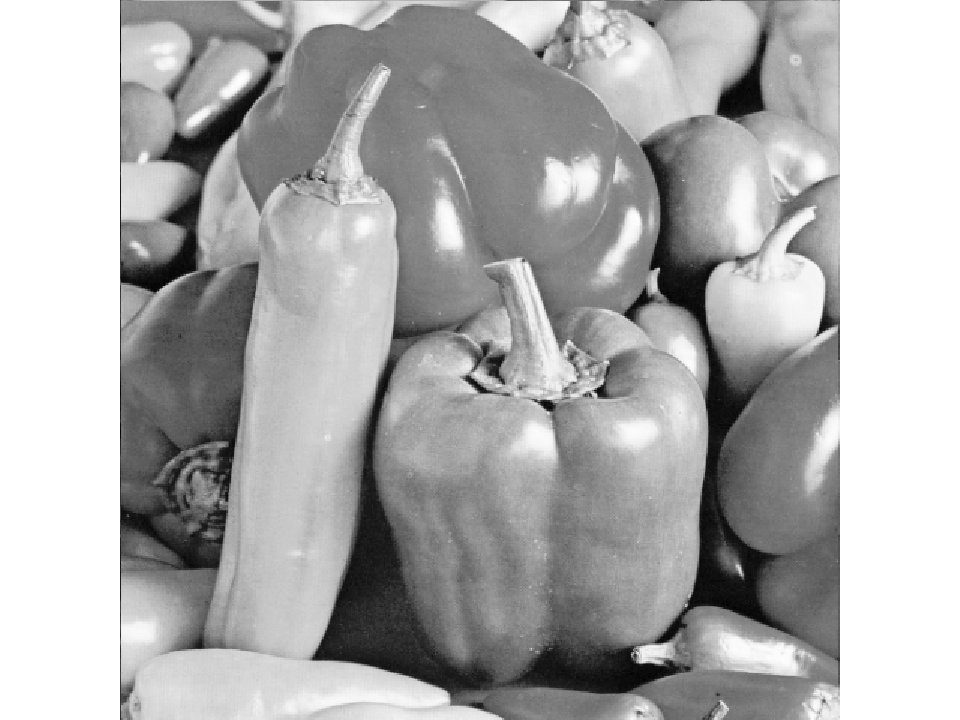} &
				\includegraphics[width=0.23\textwidth]{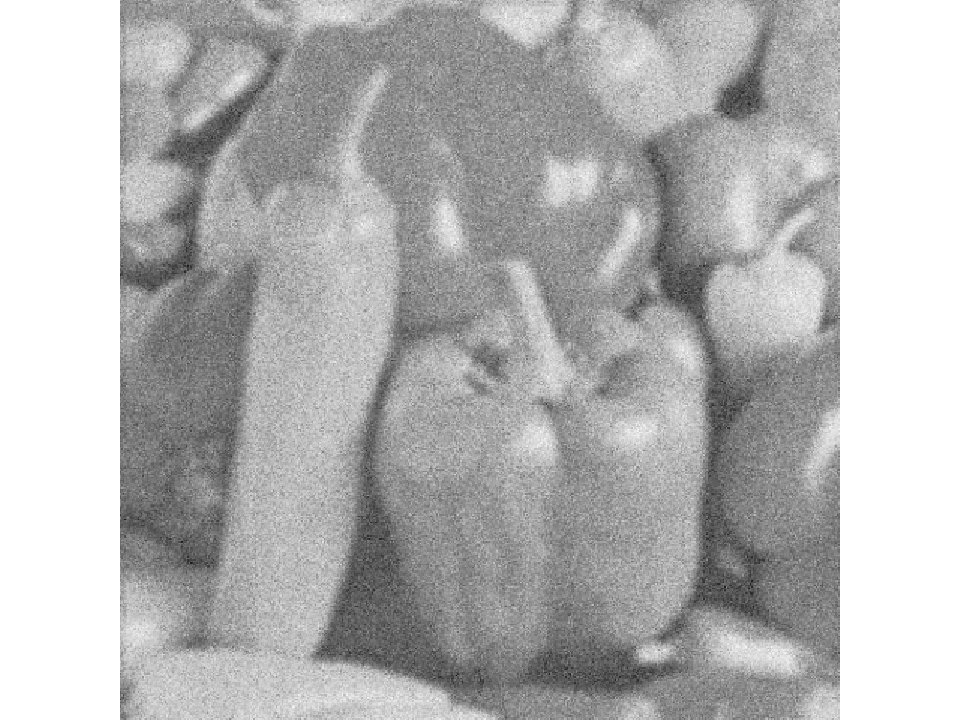} &
				\includegraphics[width=0.23\textwidth]{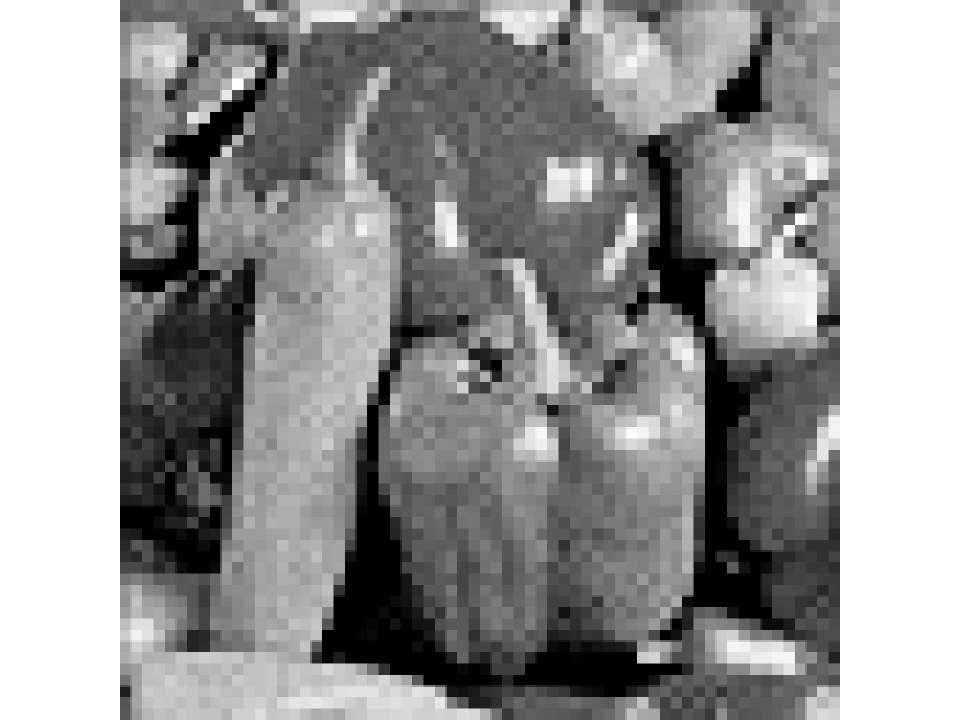} &
				\includegraphics[width=0.23\textwidth]{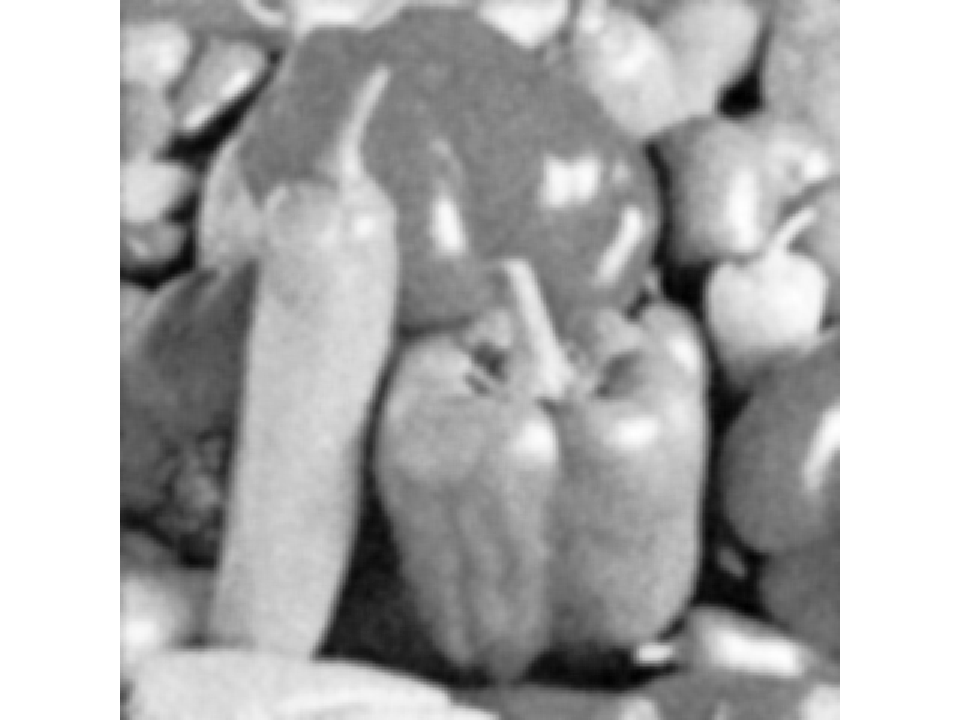}
				\\
				
				\includegraphics[width=0.23\textwidth]{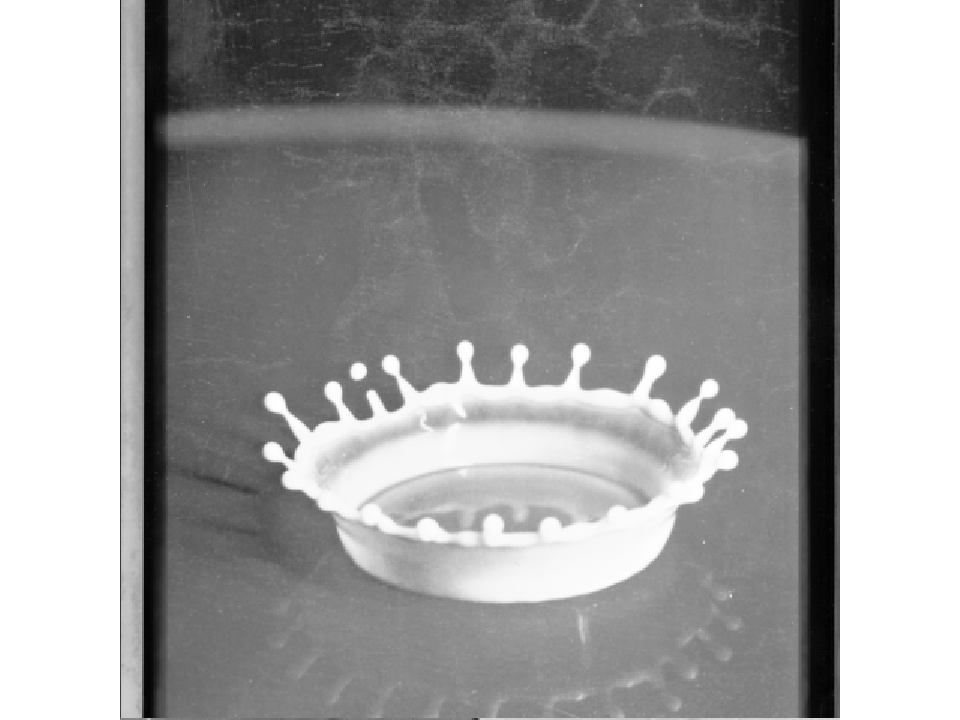} &
				\includegraphics[width=0.23\textwidth]{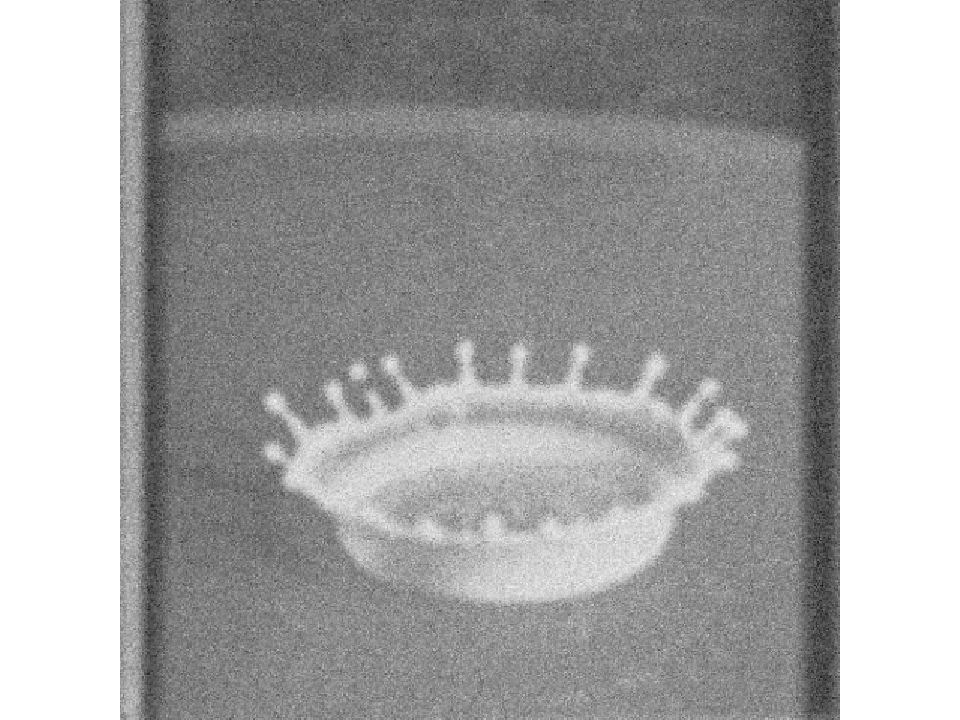} &
				\includegraphics[width=0.23\textwidth]{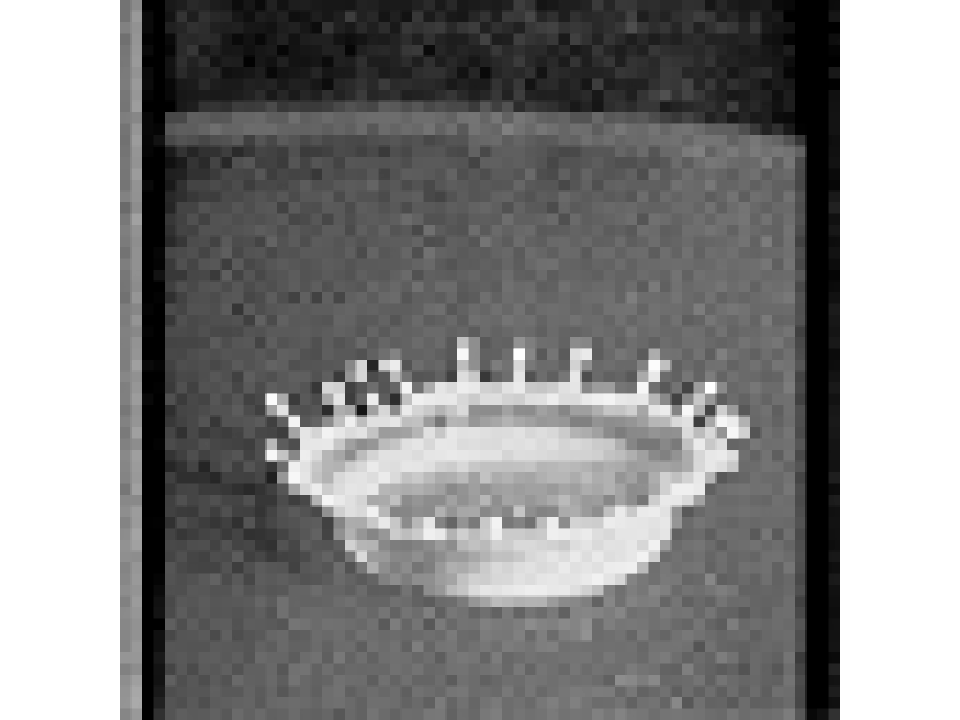} &
				\includegraphics[width=0.23\textwidth]{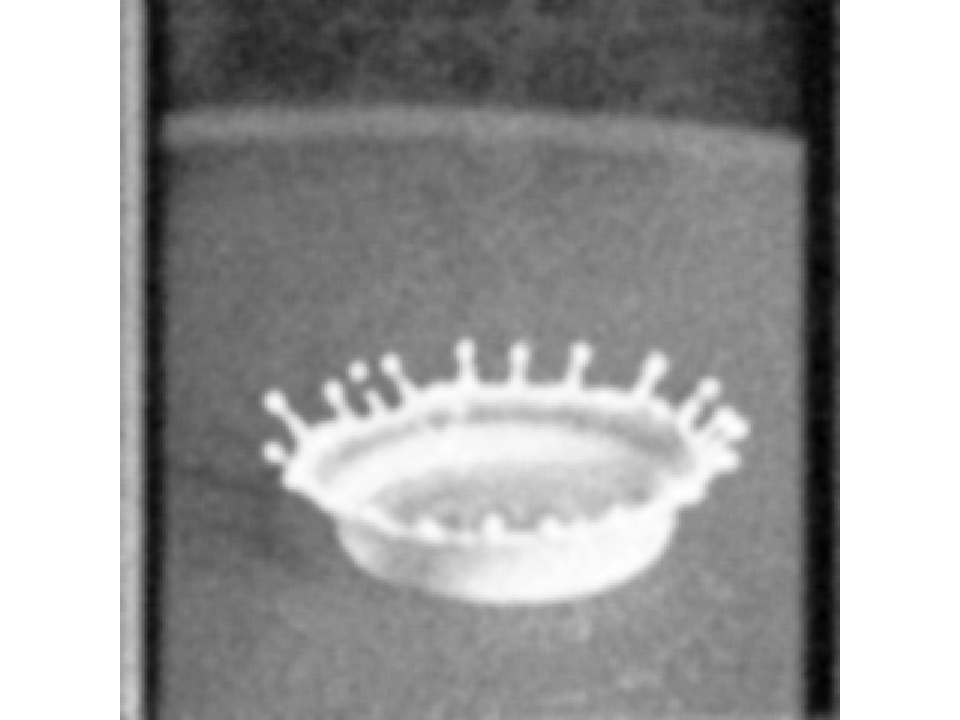}
				\\
				
				\includegraphics[width=0.23\textwidth]{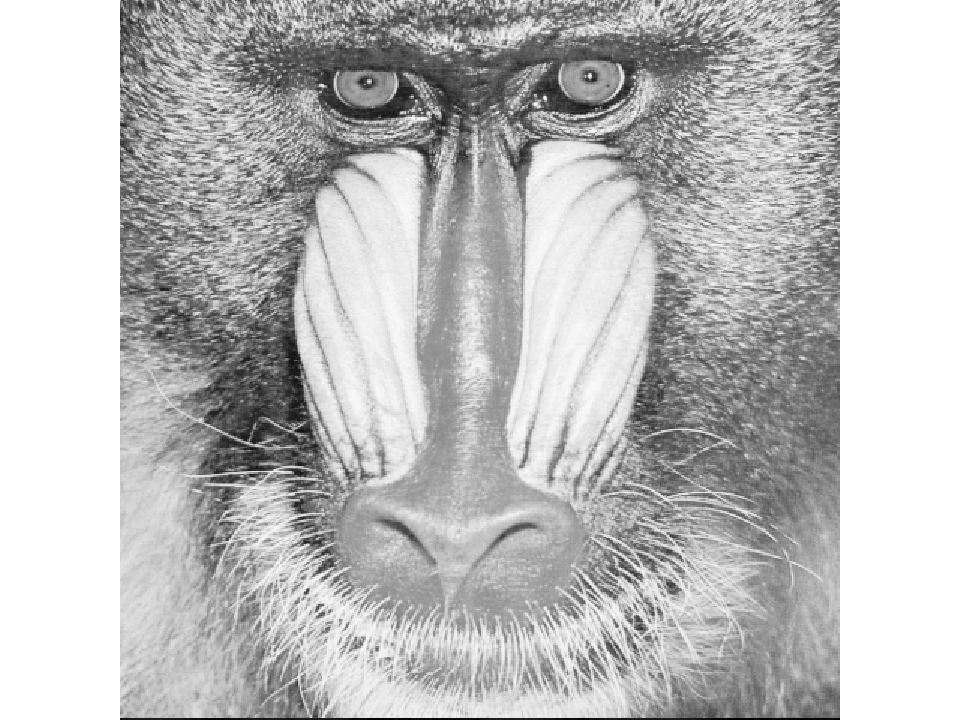} &
				\includegraphics[width=0.23\textwidth]{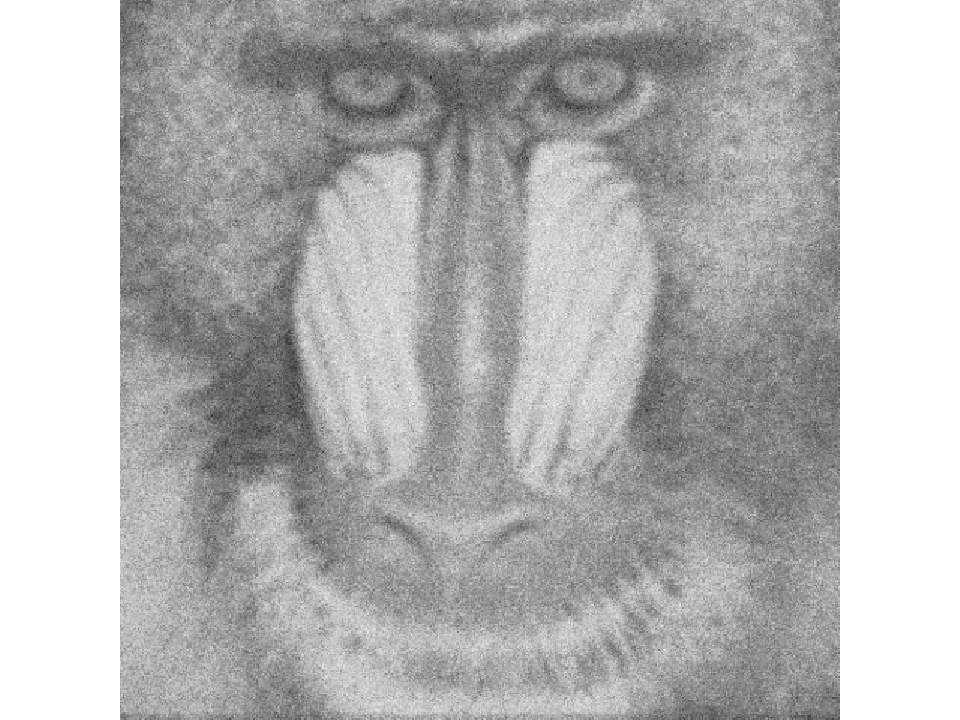} &
				\includegraphics[width=0.23\textwidth]{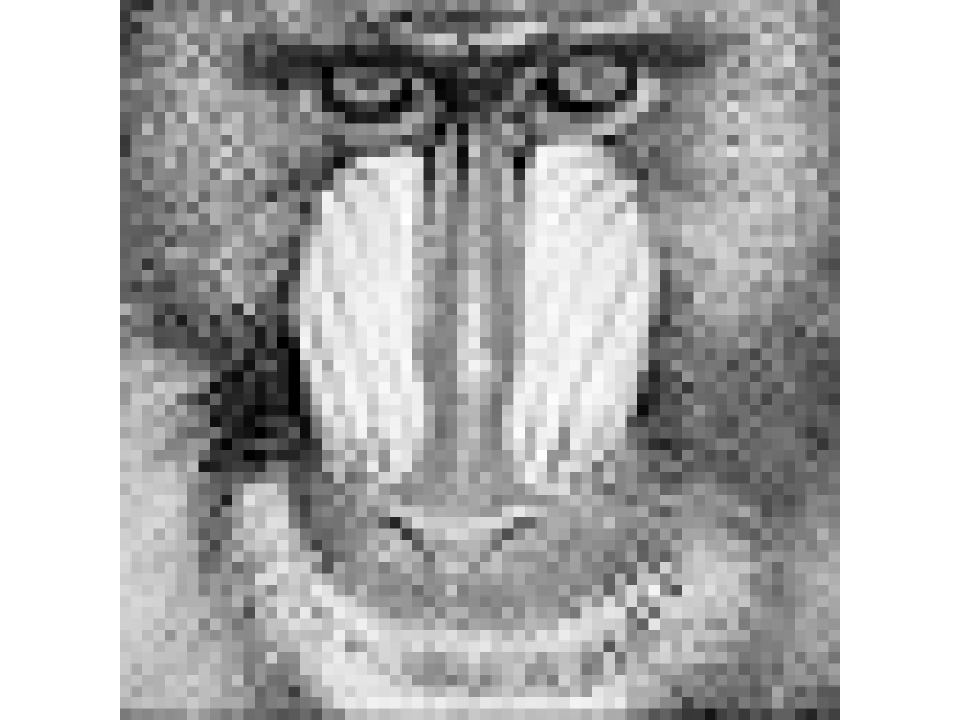} &
				\includegraphics[width=0.23\textwidth]{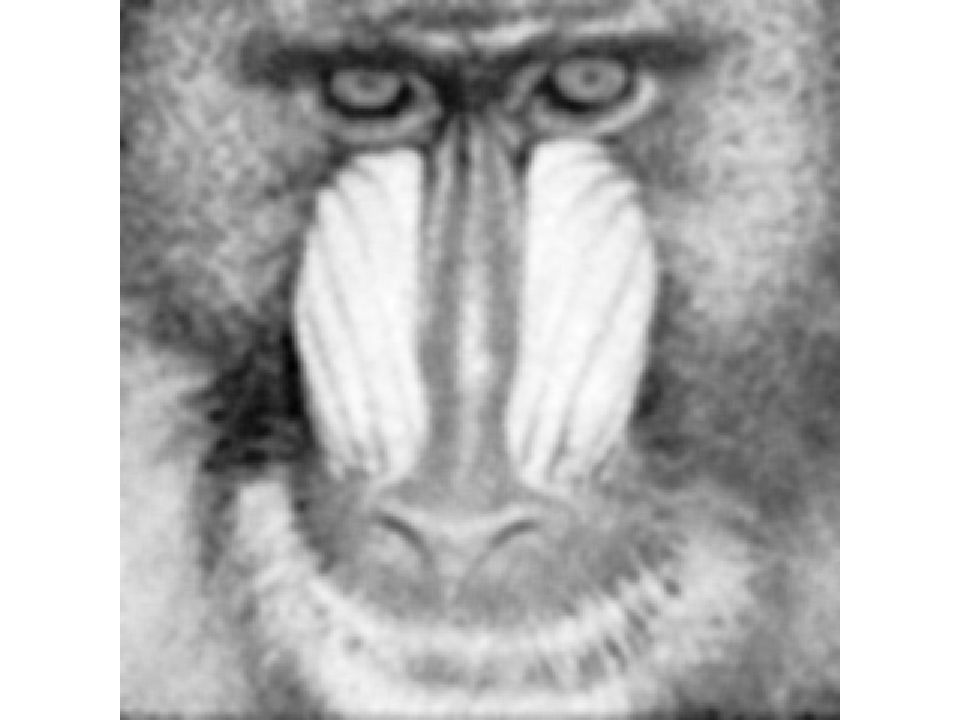}
			\end{tabular}
			
			\caption{
				Denoising results for \texttt{peppers}, \texttt{splash}, and \texttt{mandrill} images.
				The noisy observations were generated by adding Gaussian noise with
				standard deviation $\sigma=10^{-1}$.
				The third and fourth columns show the recovered images obtained with
				regularization parameters $\omega=1$ and $\omega=10$, respectively.
			}
			\label{fig:denoise}
		\end{figure}
		
		For one of the pictures \texttt{(mandrill)}, we plot the optimality values given by the quantities 
		\[  \Phi_k = F(x_k)-F(x^*)-\inner{\nabla F(x^*)}{x_k-x^*}, \qbox{and} \Gamma_k =g^*(y_k)-g^*(y^*)-\inner{\mathcal{M}x^*}{y_k-y^*},\]
		which vanish thanks to Theorem \ref{T:implicit_inexact} and \eqref{eq:decay_Phi_Gamma}. We also plot the relative error given by 
		\[\mathcal{R}:(z_{k+1},z_{k}) \mapsto \dfrac{\norm{z_{k+1}-z_{k}}}{\norm{z_{k}}},\]
		where $z_k = (x_k,y_k)$. Figure \ref{fig:mandril_plots} shows the plot of the previous measures for Algorithms \ref{alg:alg_xkplus1_wavelets} and \ref{alg:alg_xk_wavelets} after 300 iterations. 
		\begin{figure}[h]
			\centering
			\subfigure[Optimality values, $\omega=1$]{\includegraphics[width=0.48\textwidth]{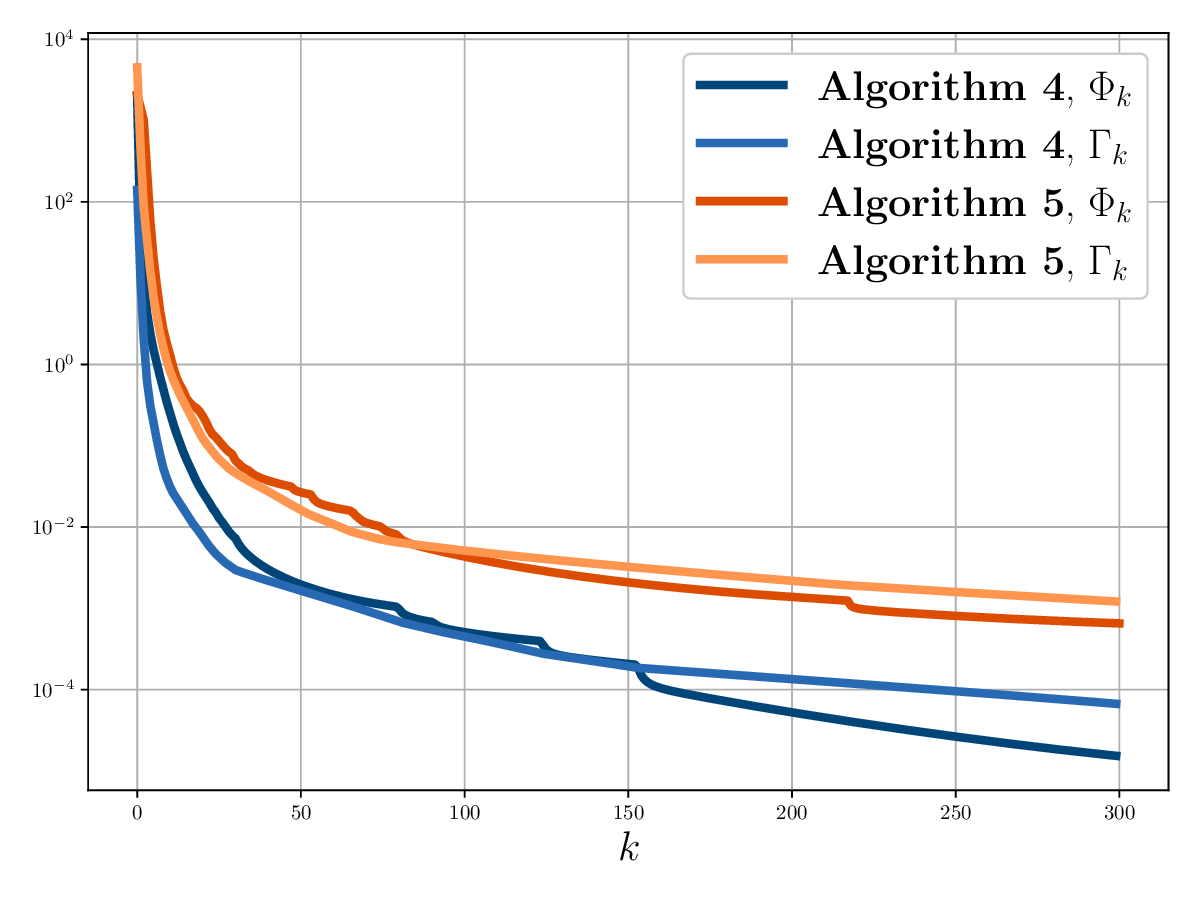}}
			\subfigure[Relative error, $\omega=1$]{\includegraphics[width=0.48\textwidth]{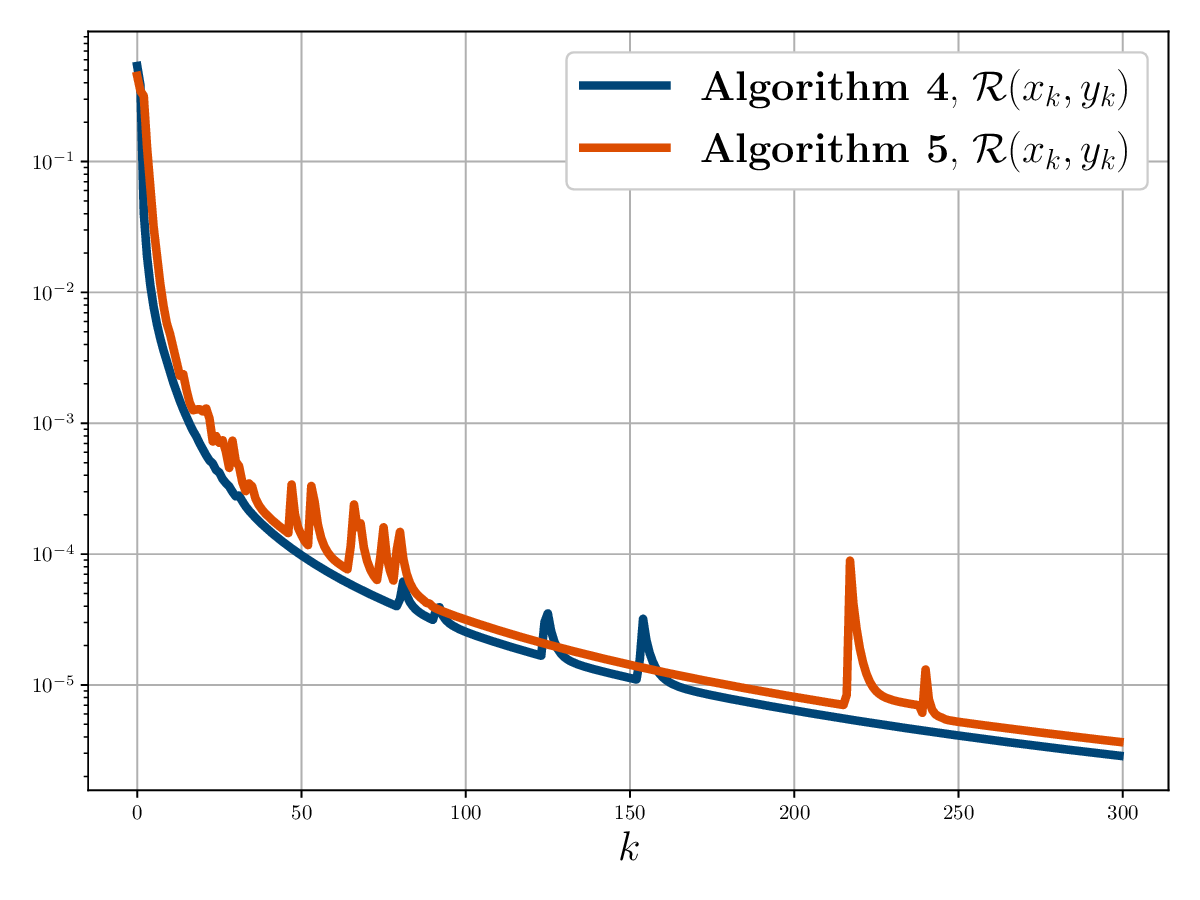}}\\
			\subfigure[Optimality values, $\omega=10$]{\includegraphics[width=0.48\textwidth]{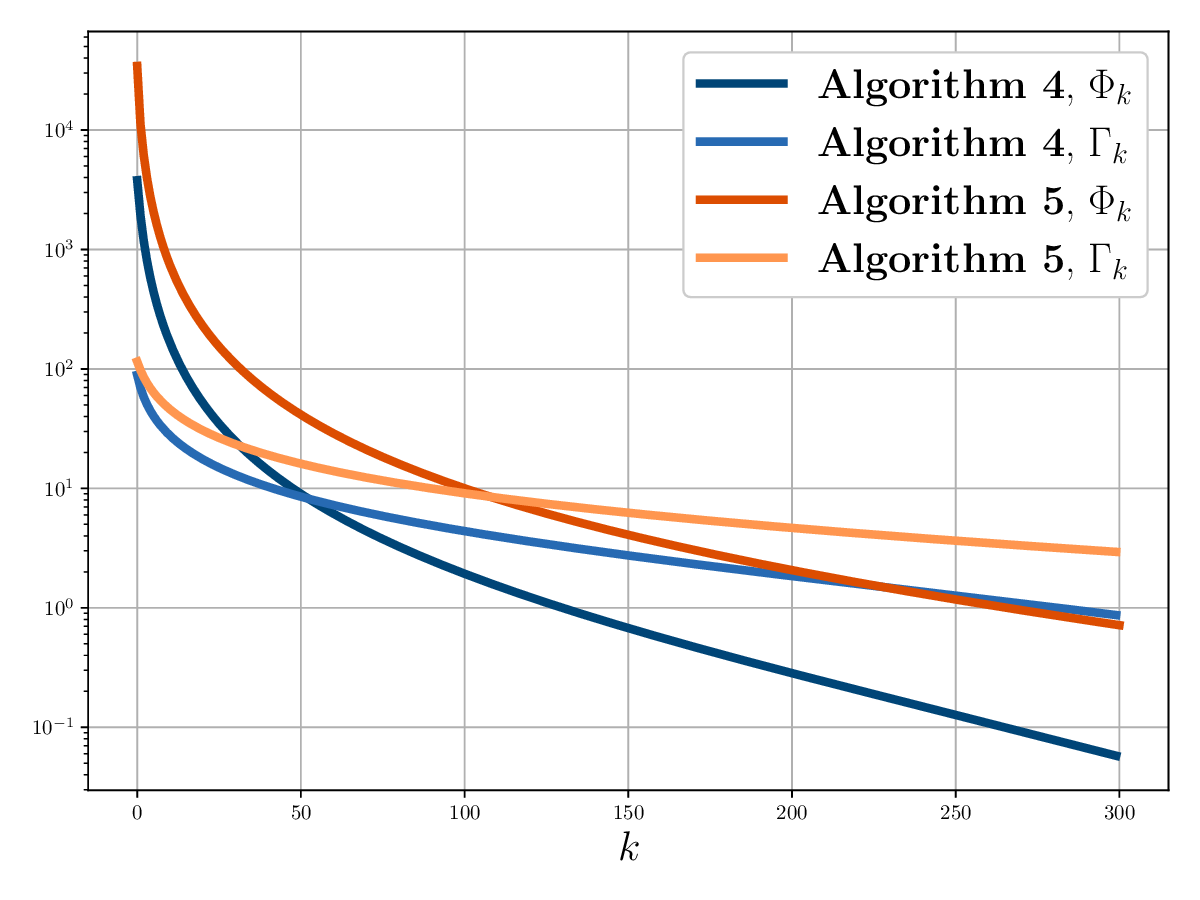}}
			\subfigure[Relative error, $\omega=10$]{\includegraphics[width=0.48\textwidth]{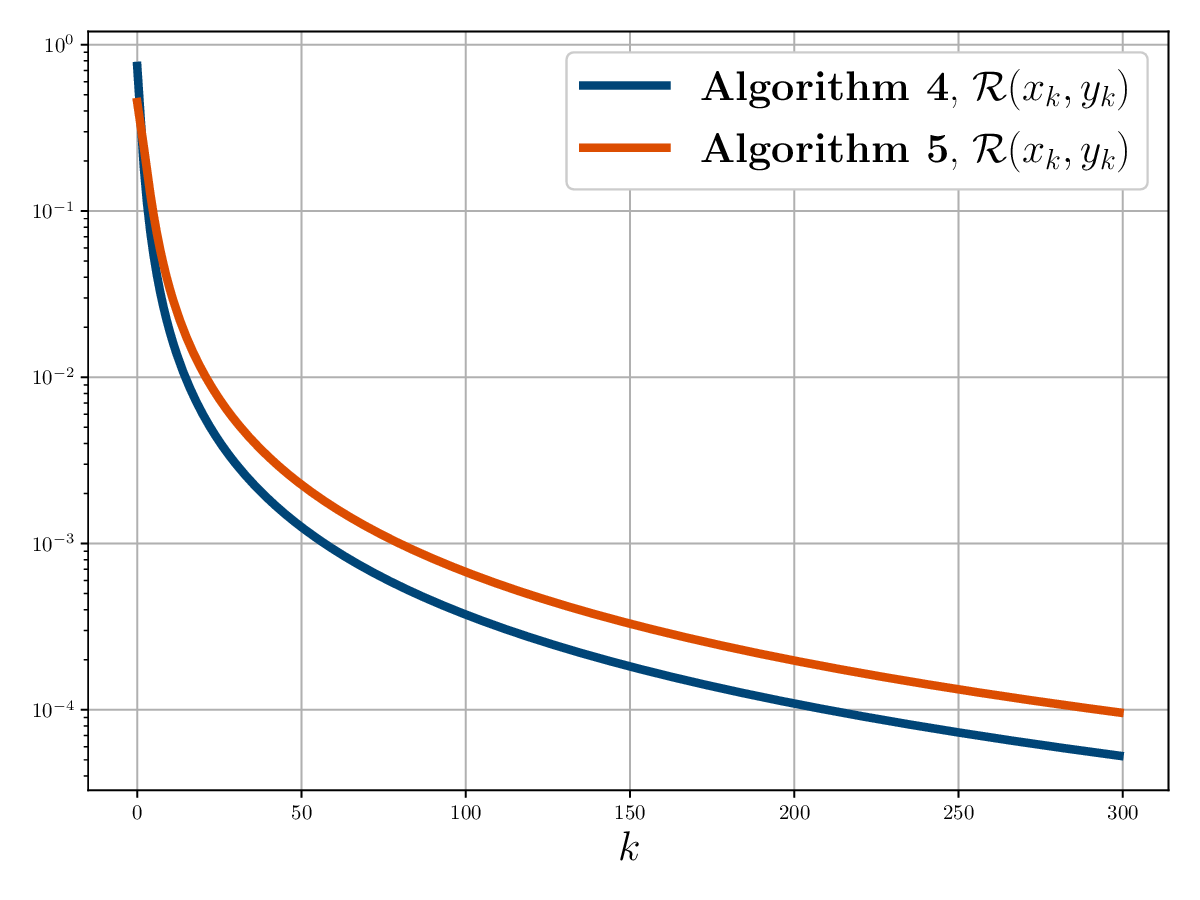}}
			\caption{Optimality measures and relative error for \texttt{mandrill}, 300 iterations. }
			\label{fig:mandril_plots}
		\end{figure}
		\section*{Acknowledgments}
		This research benefited from the support of the FMJH Program Gaspard Monge
		for optimization and operations research and their interactions with data science.  The second author was supported by ANID-Chile, Fondecyt Postdoctorado Grant 3250609, and Centro de Modelamiento Matemático (CMM) BASAL fund FB210005 for centers of excellence.
		\bibliographystyle{abbrv}
		\bibliography{references}

		\appendix
		\renewcommand{\thesection}{\Alph{section}}
		
		\section{Appendix}
		
		\subsection{Derivation of the function $H_k$}\label{appendix:function_H}
		Let us recall that
		\[H_{k}(y)=\delta_{k}y-\tilde\beta_{k}A\prox_{\eta_{k}f}\left(v_{k}-\eta_{k}\tilde\beta_kA^*y\right)-w_{k}.\]
		\begin{lemma}
			Let 
			\begin{align*}
				\mathcal{H}_{k}(y):=&\frac{\delta_{k}}{2}\|y\|^2-\langle w_{k},y\rangle+f^*\left(\prox_{f^*/\eta_{k}}(v_{k}/\eta_{k}-\tilde\beta_kA^*y)\right)\\
				&+\frac{1}{2\eta_{k}}\left\|\prox_{\eta_{k}f}\left(v_{k}-\eta_{k}\tilde\beta_kA^*y\right)\right\|^2.    
			\end{align*} 
			It holds that $\nabla \mathcal{H}_{k}(y) = H_k(y)$. 
		\end{lemma}
		\begin{proof}
			First, let us recall the Moreau identity: for every $\lambda >0 $ and every convex function $\phi$, it holds
			\[\prox_{\lambda \phi} (x)  + \lambda \prox_{\frac{\phi^*}{\lambda}}\left( \frac{x}{\lambda}\right) = x .\]
			Rescaling, we obtain
			\begin{equation}\label{eq:moreau_prox}
				\prox_{\lambda \phi} (\lambda x) = \lambda \left( x - \prox_{\frac{\phi^*}{\lambda}}\left( x\right) \right).
			\end{equation}
			Let us recall the definition of the \textit{Moreau envelope} of a convex function $\phi$ with parameter $\lambda > 0$:  
			\[\operatorname{e}_{\lambda}\phi(u) =\min_{z} \left\lbrace \phi(z)  + \frac{1}{2\lambda}\Vert z - u\Vert^2 \right\rbrace,\]
			where the minimizer is attached exactly at $z= \prox_{\lambda \phi} (x)$. Let us recall that the envelope is continuously differentiable, and its gradient can be computed as 
			\[\nabla \operatorname{e}_{\lambda}\phi(u) = \frac{1}{\lambda} \left( u - \prox_{\lambda \phi} (u)\right).\]
			Notice that the last two terms in the definition of $\mathcal{H}_k$ can be expressed as a Moreau envelope using \eqref{eq:moreau_prox}, meaning 
			\begin{align*}
				\operatorname{e}_{\frac{1}{\eta_k}}f^*(u) &= f^*\left( \prox_{\frac{1}{\eta_k}f^*}(u) \right) + \frac{\eta_k}{2}\Vert \prox_{\frac{f^*}{\eta_k}}(u) - u\Vert^2 \\
				&= f^*\left( \prox_{\frac{1}{\eta_k}f^*}(u) \right) + \frac{1}{2\eta_k}\Vert \prox_{\eta_k f}(\eta_k u) \Vert^2.
			\end{align*}
			Then, 
			\begin{equation}\label{eq:Hk_envelope}
				\mathcal{H}_k(y) = \frac{\delta_{k}}{2}\|y\|^2-\langle w_{k},y\rangle + \operatorname{e}_{\frac{1}{\eta_k}}f^*(q_k),
			\end{equation}
			with $q_k = \frac{v_k}{\eta_k} - \tilde\beta_kA^* y$. Using the gradient formula for the envelope and \eqref{eq:moreau_prox}, we get
			\[\nabla \operatorname{e}_{\frac{1}{\eta_k}}f^*(u) = \eta_k \left( u - \prox_{\frac{f^*}{\eta_k}}(u)\right) = \prox_{\eta_k f}\left( \eta_k u\right).\]
			Notice that since $q_k$ depends on $y$, when calculating the gradient of $\mathcal{H}_k$ we must consider also $\nabla_yq_k = -\tilde\beta_kA$. Collecting the previous, computing the gradient in \eqref{eq:Hk_envelope} matches the exact definition of $H_k(y)$. 
		\end{proof}
	\end{document}